\documentclass[12pt]{article}
\usepackage[left=1in, right=1in, top=1in, bottom=1in,,headheight=57pt,headsep=0in,includehead]{geometry}
\usepackage{setspace}
\usepackage{amsfonts,amsmath,amsthm}
\usepackage{comment}
\newtheorem{theorem}{Theorem}
\usepackage{accents}

 \usepackage{algorithm} 
 \newenvironment{varalgorithm}[1]
  {\algorithm\renewcommand{\thealgorithm}{#1}}
  {\endalgorithm}

\usepackage{latexsym}
\usepackage{amssymb,amsmath,amsfonts}
\usepackage{graphicx}
\usepackage{enumitem}
\usepackage{xcolor}
\usepackage[colorlinks=true,hypertexnames=false]{hyperref}
\newtheorem{lemma}{Lemma}
\newtheorem{cor}{Corollary}

\theoremstyle{definition}
\newtheorem{definition}{Definition}
\newtheorem{ass}{Assumption}

\newtheorem{prop}{Proposition}

\theoremstyle{remark}
\newtheorem{remark}{Remark}

\newcommand*{\norm}[1]{\|#1\|}
\DeclareMathOperator{\dist}{dist}

\newcommand*{\inner}[2]{\langle#1,#2\rangle}
\DeclareMathOperator*{\dom}{dom}
\DeclareMathOperator*{\Lev}{Lev}
\newcommand*{\lev}[2]{{\Lev\nolimits}_{{#1}} {(#2)}}
\newcommand{\ubar}[1]{\underaccent{\bar}{#1}}
\DeclareMathOperator*{\argmin}{\arg\!\min}
\DeclareMathOperator*{\Prox}{prox}
\newcommand*{\prox}[2]{\Prox\nolimits_{#1}(#2)}
\newcommand{\R}{\mathbb{R}}
\newcommand{\N}{\mathbb{N}}

\allowdisplaybreaks

\title{On the Convergence Rates of Iterative Regularization Algorithms for Composite Bilevel Optimization\thanks{This research was supported by the Israel Science Foundation (ISF), grant no.~1603/23.}}

\author{Shimrit Shtern\thanks{Faculty of Data and Decision Sciences, Technion -- Israel Institute of Technology, Haifa 3200003, Israel (\href{mailto:shimrits@technion.ac.il}{shimrits@technion.ac.il}).}
\and Adeolu Taiwo\thanks{Department of Mathematics, Arizona Western College, Yuma, AZ 85365, United States (\href{mailto:adeolu.taiwo@azwestern.edu}{adeolu.taiwo@azwestern.edu}).}}
\date{}

\usepackage{amsopn}

\begin{document}

\maketitle

\begin{abstract}
This paper investigates iterative methods for solving bilevel optimization problems where both inner and outer functions have a composite structure. We establish novel theoretical results, including the first  analysis that provides simultaneous convergence rates for the {Iteratively REgularized} Proximal Gradient (IRE-PG) method, a variant of Solodov's algorithm. These rates for the inner and outer functions highlight the inherent trade-offs between their respective convergence behaviors. We further extend this analysis to an accelerated version of IRE-PG, proving faster convergence rates under specific settings. Additionally, we propose a new scheme for handling cases where these methods cannot be directly applied to the bilevel problem due to the difficulty of computing the associated proximal operator. This scheme offers surrogate functions to approximate the original problem and a framework to translate convergence rates between the surrogate and original functions. Our results show that the accelerated method’s advantage diminishes under this translation.\end{abstract}



\section{Introduction}\label{Sec1}

 Bilevel optimization problems consist of solving an optimization problem (the outer level) where some of the variables are constrained to the solution set of another optimization problem (the inner level). This framework serves as a model for myriad applied problems in diverse fields such as decision sciences, machine learning, regression analysis, and signal processing. A review of bilevel optimization problems and their applications can be found in \cite{dempezemkoho2020, Sin}.
 
    The simple convex bilevel optimization problem consists of solving an outer-level problem given by
\begin{equation}\label{Bilevel}\tag{BLO}
	 	\min_{x\in X^*}  \omega(x),
    \end{equation}
     where the outer function $\omega$ is convex, and $X^*$ is the nonempty solution set of an inner-level problem  \begin{equation*}%
    \min_{x\in \R^n}\varphi(x),
    \end{equation*}
    where the inner function $\varphi$ is also convex.
    An inherent difficulty associated with solving problem \eqref{Bilevel} is that the problem does not satisfy regularity conditions, so duality or KKT-based methods cannot be applied. 

First-order methods are iterative techniques that rely only on function values and their (sub)gradients. They are widely used in machine learning and signal processing due to their low memory usage and computational cost per iteration, making them particularly suitable for solving large-scale problems. See the book \cite{beck2017} and the recent survey \cite{DvuShtSta2021} for an overview of first-order methods for convex optimization. Specifically, in this work, we focus on first-order methods to solve problem \eqref{Bilevel}.

A popular approach to solve problem \eqref{Bilevel} is to iteratively solve the regularized minimization problem 
   \begin{equation}
      \min_{x\in \R^n}\{\phi_{k}(x):= \varphi(x)+\sigma_k\omega(x)  \}\label{Tikhonov}
  \end{equation}
  for some well-chosen sequence $\{\sigma_k\}_{k\in\N}$.
 This method is reminiscent of Tikhonov-type regularization suggested for solving an underdetermined set of linear equalities \cite{tikhonov1963solution,Tikhonov:1977} where $\omega$ was set as the $\ell_2$ norm squared. %
For a comprehensive overview of early work on the iterative regularization method, the reader is referred to \cite{tikhonov_numerical_1995,tikhonov_nonlinear_1998}, while the continuous version of this approach is studied in \cite{alber_continuous_1968, amochkina1997continuous, vasil1997continuous}. %
 Friedlander and Tseng \cite{friedlander_exact_2008} identify conditions guaranteeing a bound $\bar{\sigma}$ such that solving \eqref{Tikhonov} with any $\sigma_k \leq \bar{\sigma}$ produces exact solutions to \eqref{Bilevel}; in practice, however, this threshold is usually unknown.   

Since solving problem~\eqref{Tikhonov} exactly for various values of $\sigma_k$ is computationally demanding, iterative algorithms have been proposed to approximate the solution of $\phi_k$ at each iteration $k$. The origin of such algorithms can be traced back to the work of Cabot \cite{cabot2005}, who introduced an iterative proximal-type algorithm for $\phi_k$ and established its convergence to the solution set of problem~\eqref{Bilevel} in the case where $\omega$ is a real-valued function and $\varphi$ is a closed and convex function. Later, Solodov \cite{Sol} proposed an explicit first-order descent method under the additional assumption that $\varphi(x)=f(x)+\delta_X(x)$, where $f$ and $\omega$ are smooth with locally Lipschitz continuous gradients, and $\delta_X$ is the indicator function of a closed and convex set $X$, called hereinafter the smooth case. This method replaces the proximal step in \cite{cabot2005} with a projected gradient step and establishes convergence for a regularization sequence $\{\sigma_k\}_{k\in\N}$ that satisfies $\lim_{k\to \infty}\sigma_k = 0$ and $\sum_{k=1}^{\infty}\sigma_k = +\infty$. Although both methods proved convergence, they did not provide explicit convergence rates.

 In the past decade, first-order algorithms with proven convergence rates have been proposed for the case of strongly convex $\omega$. 
 For strongly convex $\omega$,   Beck and Sabach \cite{becsab2014} and later Sabach and Shtern \cite{sabachshtern2017} provided algorithms with proven inner function convergence rates of $\mathcal{O}(1/\sqrt{k})$ and $\mathcal{O}(1/{k})$, respectively. 
  In both algorithms, the function $\varphi$ is assumed to be a composite function, that is, $\varphi=f+g$, where $f$ is a convex and differentiable function with an $L_f$-Lipschitz continuous gradient ($L_f$-smooth), and  $g$ is a proper, closed, and convex function.
  Amini and Yousefian \cite{amiyou2019} presented the IR-PG method, an iterative regularization method %
  for the case where both $\omega$ and $\varphi$ are {possibly nondifferentiable}, which converges to the solution of problem \eqref{Bilevel} with a rate of $\mathcal{O}(1/k^{0.5-\beta})$ for the inner-level problem, where $\beta\in (0, 0.5)$. 
  
First-order algorithms with proven convergence rates for both the inner and outer functions have only been introduced within the last few years.
 In \cite{KaushikYousefian2021}, Kaushik and Yousefian proposed an iteratively regularized (sub)gradient method for solving 
 the problem of minimizing a convex $\omega$ with bounded subgradients over a given closed and convex set $X$, where the solution has to also 
 satisfy a monotone variational inequality defined over $X$. Their problem structure includes problem~\eqref{Bilevel}  where $\varphi=f+\delta_X$, and $f$ is {$L_f$-}smooth and has bounded gradients over $X$.  The authors proved  convergence rates of $\mathcal{O}(1/k^{0.5-\beta})$ and $\mathcal{O}(1/k^\beta)$ for the outer and inner functions, respectively, where $\beta\in(0,0.5)$. Similarly, assuming $f$ is convex, $X$ is compact, and $\omega$ is convex and smooth, Giang-Tran, Ho-Nguyen, and Lee \cite{giang-tran2024} proposed using iterative regularization paired with the conditional gradient algorithm and showed $\mathcal{O}(1/k^{1-\beta})$ and $\mathcal{O}(1/k^\beta)$ convergence rates for the outer and inner functions, respectively, where $\beta\in(0,1)$, and improved these rates under additional structural assumptions.  
 Other first-order algorithms, which are not based on iterative regularization, also proved convergence with regard to both inner and outer functions. Specifically, Shen, Ho-Nguyen, and K{\i}l{\i}n\c{c}-Karzan \cite{shen2023online} presented online convex optimization algorithms, \cite{doronshtern2023} presented a nested algorithm,  \cite{merchavshoham2023} presented an alternating proximal gradient (PG) algorithm, and \cite{jiangetal2023} presented a conditional gradient-based algorithm for solving {problem \eqref{Bilevel}} under various assumptions on functions $\varphi$ and $\omega$.
 Specifically, Merchav and Sabach \cite{merchavshoham2023} provided  
 $\mathcal{O}(1/k^{\beta})$ and $\mathcal{O}(1/k^{1-\beta})$ convergence rates for the inner and outer functions, respectively, for any  $\beta\in(0.5,1)$, which are the best rates achieved under the assumptions that both $\varphi$ and $\omega$ are composite functions (the composite case), $\omega$ is real-valued and coercive, and the method uses only PG operators on $\omega$ and $\varphi$. Recently, \cite{latafat2024} suggested several iterative regularization schemes for the composite case; however, they did not provide convergence guarantees to a solution of the bilevel problem, or rates on inner and outer function values. We note that after the first submission of the current manuscript, Bo\c{t} et al. \cite{bot_accelerating_2025} presented an analysis of an iterative regularization algorithm for the composite case, with their main convergence-rate results obtained under further geometric assumptions on the inner-level function, and a discussion of extensions under milder assumptions. 
To the best of our knowledge, no prior work has established simultaneous convergence-rate guarantees for the iterative regularization scheme in the composite case under the standing assumptions considered here.

Additionally, during the last few years, accelerated first-order-based schemes for solving problem~\eqref{Bilevel} have been gaining attention. Initially, only convergence proofs without rates were available for such algorithms \cite{HelouSim2017, shehuvuongzemkoho2021, duanzhang2023}. 
In this work, we focus on algorithms with proven convergence rates, specifically those that do not require weak-sharpness or H\"olderian error bound assumptions on $\varphi$---assumptions that have been utilized to obtain {improved}  convergence rates in works such as \cite{samadietal2023,cao2024accelerated,chenetal2024,merchav2024,bot_accelerating_2025}. Without such additional  assumptions, Samadi, Burbano, and Yousefian \cite{samadietal2023} proposed a regularized accelerated proximal method for the case of Lipschitz smooth $\omega$ and composite $\varphi$  
and proved that for a 
predefined iteration number $K>0$, taking the regularization parameter as $(K+1)^{-1}$ an optimality gap of $\mathcal{O}(1/K)$ is obtained for both the inner and outer functions.  
Cao et al. \cite{cao2024accelerated} suggested an accelerated projected gradient method for the smooth setting with compact set $X$ and proved $\mathcal{O}(1/k^2)$ inner and $\mathcal{O}(1/k)$ outer convergence rates. Under a similar smooth setting, for a predefined $\epsilon$, Zhang et al. \cite{zhang_functionally_2024}  combined the generalized accelerated gradient method \cite[Scheme (2.3.13)]{nesterov_lectures_2018} with bisection to obtain an $\mathcal{O}(1/\sqrt{\epsilon}\log(1/\epsilon))$ iteration complexity for both inner and outer functions.
 For the case where both  $\omega$ and $\varphi$ 
  are composite functions,  Wang, Shi, and Jiang \cite{wang2024near} proposed a bisection method for predefined $\epsilon_\varphi$ and $\epsilon_\omega$ that requires $\mathcal{O}(1/\sqrt{\epsilon_\varphi}\log(1/\epsilon_\omega))$ iterations. 
 Finally,  \cite{merchav2024} proposed a fast bilevel PG algorithm, an iterative regularization approach. By choosing the regularization parameter as $\sigma_k=1/(k+a)$ for some $a\geq 2$ the authors proved 
 simultaneous rates of $\mathcal{O}(\ln(k+1)/k)$ and $\mathcal{O}(1/k)$ for the inner and outer functions, respectively. Notably, the algorithms that do not require either a predefined number of iterations or a predefined target optimality gap,  namely \cite{cao2024accelerated} and \cite{merchav2024}, necessitate computing a potentially  computationally expensive oracle at each iteration. Specifically, \cite{cao2024accelerated} requires a projection over an intersection of two simple sets, and \cite{merchav2024} requires computing the proximal operator on the sum of two proximal friendly functions. However, the computational complexity of these operations is not considered in their analyses. In this work, we investigate an accelerated iterative regularization method, similar to the one presented in  \cite{merchav2024}, and analyze the convergence rate for composite $\varphi$ and $\omega$, taking into account that proximal operators can only be applied to the nonsmooth part of the inner and outer functions separately. 

In summary, this paper aims to provide a better understanding of iterative regularization methods for solving problem \eqref{Bilevel} when both the inner and outer functions have a composite structure. Specifically, the contributions of this paper are as follows:
\begin{enumerate}[label=\arabic*.]
    \item {\bf Convergence Rate Analysis for Iterative Regularization Proximal Gradient Method:}
To the best of our knowledge, we are the first to establish simultaneous convergence rates for the iterative regularization PG method, with both constant and backtracking-based step sizes. This algorithm is a variant of Solodov’s method \cite{Sol}, which only deals with composite inner-level problems, and thus closes a gap in the existing literature.
Our convergence rate analysis provides insights into the trade-off between the convergence rates of the inner and outer functions. Specifically, when the regularization parameter is set to $k^{-\beta}$, we obtain convergence rates of $\mathcal{O}(k^{-\beta})$ for the inner function and $\mathcal{O}(k^{-(1-\beta)})$ for the outer function for any $\beta \in (0,1)$. These results are comparable to the rates achieved by the algorithms proposed in \cite{merchavshoham2023,giang-tran2024}.
    \item {\bf Convergence Rate Analysis for the Accelerated Algorithm}:
We establish simultaneous convergence rates for an accelerated version of the iterative regularization PG method for both constant and backtracking-based step sizes. In particular, we obtain a rate of $\mathcal{O}(k^{-(2-\beta)})$ for the outer function and $\mathcal{O}(k^{-\beta})$ for the inner function, for any $\beta\in(0,1)$, and when 
$\beta=1$, the rates include an additional logarithmic term for the inner function. Thus, we illustrate the trade-off in this algorithm between the convergence rates of the inner and outer functions. These convergence rates coincide with the rates obtained in \cite{merchav2024} for 
$\beta=1$ and extend them to a broader range of $\beta$, while eliminating the need to bound the smoothness parameters. Moreover, while \cite{merchav2024} demonstrates the possibility of obtaining a superior convergence rate for the inner function, our results suggest that this cannot be achieved simultaneously with a guaranteed rate for the outer function without additional assumptions.
    \item  {\bf Extension Scheme for Intractable Proximal Operators:}
We propose a scheme to address cases where PG-based iterative methods cannot be directly applied to problem \eqref{Bilevel} due to the difficulty of computing a proximal operator. This scheme not only introduces surrogate functions that can be used in place of 
$\varphi$ and $\omega$, but also provides a translation mechanism to relate the convergence rates of the surrogate functions back to the original problem. Analyzing our convergence results under this scheme, we show that the advantage of the accelerated method diminishes due to this translation. Specifically, if $\omega$ is real-valued, a convergence rate of $\mathcal{O}(\log(k)k^{-1})$ can be maintained for the inner function, but the outer function’s rate decreases to $\mathcal{O}((\log(k)/k)^{0.5})$. Conversely, if $\varphi$ is real-valued, we can only achieve a convergence rate of 
$\mathcal{O}((\log(k)/k)^{0.5})$ for the inner function while preserving a rate of 
$\mathcal{O}(k^{-1})$ for the outer function.
\end{enumerate}

{The paper is structured as follows.}
Section~\ref{Sec2} provides the necessary preliminaries, including assumptions on the model and the algorithms used in this paper, along with proofs for some technical lemmas that are foundational to our analysis. In Section \ref{sec:rate_translate}, we introduce the scheme to handle cases where the proposed algorithms cannot be directly applied to problem~\eqref{Bilevel} and discuss the corresponding rate translation. %
Section~\ref{sec:four} presents the {Iteratively REgularized} Proximal Gradient algorithm (\ref{alg:IRE-PG}) for solving problem \eqref{Bilevel}, along with its convergence analysis and derivation of its convergence rates. In Section~\ref{sec:five}, we introduce the Iteratively REgularized Accelerated Proximal Gradient algorithm (\ref{alg:IRE-APG}) and establish its convergence rates. Section~\ref{sec:numetical_experiments} validates these theoretical results through numerical experiments, and Section~\ref{sec:conclusions} provides concluding remarks and future research directions.
 
\paragraph{Notation} For any $K\in\N$, we use $[K]$ to denote the set $\{1,2,\ldots,K\}$.  
$\bar{\R}$ denotes the extended real-valued set $\R\cup\{+\infty\}$, and $\N_0$ denotes the set of natural numbers including zero, {\it i.e.}, $\N\cup\{0\}$.
Additionally, for any sequence $\{a_k\}_{k\in\N_0}\subseteq\R$ and any $K<N\in\N_0$ we use the convention $\sum_{k=N}^K a_k =0$. 
For any closed and convex set $\mathcal{C}\subseteq\R^n$ and vector $x\in \R^n$, we define  $\dist(x,\mathcal{C})=\min_{y\in\mathcal{C}} \norm{x-y}$.
For any extended real-valued function $f:\R^n\rightarrow\bar{\R}$ and $c\in\R$, we define the $c$-level set of $f$ as $\lev{f}{c}=\{x\in\R^n:f(x)\leq c\}$.   

\section{Model Assumptions and Mathematical Preliminaries}\label{Sec2}
\subsection{The Composite Simple Bilevel Optimization Model}\label{sec:problem_plus_example}
In this paper, we consider the bilevel optimization problem \eqref{Bilevel} for convex composite functions. 
We define a convex composite function as follows.
 \begin{definition}\label{def:convex_composite}
     Function $\phi=f+g$ is called a convex composite function if $f:\R^n\rightarrow \R$ is convex, differentiable, with an $L_f$-Lipschitz continuous gradient ($L_f$-smooth), and $g:\R^n\rightarrow\bar{\R}$ is proper, convex, and lower semicontinuous.
\end{definition}
Thus, our analysis is done under the following {assumptions}.
\begin{ass}\label{ass:composite}
     $\omega$ and $\varphi$ are convex composite functions. Specifically, $\omega(x)=f_1(x)+g_1(x)$ and $\varphi(x)=f_2(x)+g_2(x)$, such that for $i=1,2,$ \begin{enumerate}[label=(\roman*)]
\item
     $f_i:\R^n\rightarrow\R$ is convex, is differentiable, and has an $L_i$-Lipschitz continuous gradient.
     \item $g_i:\R^n\rightarrow \bar{\R}$ is a proper, convex, and lower semicontinuous function.
   \end{enumerate}  
 \end{ass}
 
 Moreover, in order to ensure that the {problem \eqref{Bilevel}} has an optimal solution we will further assume that $\varphi^*\equiv\min_{x\in \R^n} \varphi(x)>-\infty$ and the minimum is attained, and that $\dom(g_1)\cap\dom(g_2){\neq \emptyset}$. Finally we make the following assumption about $\omega$. %
 
 \begin{ass}\label{ass:omega_min}
   $\omega$  satisfies $$\ubar{\omega}=\min_{x\in \R^n}\omega(x)>-\infty,$$
   the optimal solution set of the bilevel problem, given by
   $\ubar{X}=\argmin_{x\in X^*}\omega(x)$ is nonempty and bounded.
   Consequently, the  minimum is attained and
    $$\omega^*=\min_{x\in X^*} \omega(x)>-\infty.$$
 \end{ass}
Assumption~\ref{ass:omega_min} can be satisfied under a variety of conditions. Specifically, 
if $\omega$  is a coercive function (as is assumed in \cite{merchavshoham2023}, for example), then %
for any $c\in\R$ the level set $\lev{\omega}{c}$
 is bounded \cite[Proposition 11.12]{bauschkecombettes2017}, and Assumption~\ref{ass:omega_min} is satisfied by
   \cite[Proposition 11.15]{bauschkecombettes2017}.  {For convenience, we denote $\Delta_\omega=\omega^*-\ubar{\omega}$.}

 Moreover, in this paper we construct algorithms based on the proximal operator given by 
 $$\prox{\lambda g}{x}=\argmin_{y\in\R^n}\left\{g(y)+\frac{1}{2\lambda}\norm{x-y}^2\right\},~\lambda >0.$$
To this end, we will assume that both $g_1$ and $g_2$ are ``proximal friendly'', that is, their proximal operator can be easily computed (see \cite[Chapter 6]{beck2017} for details of such functions).

 To motivate the structure of this bilevel problem, we consider a general inverse problem in which we wish to recover a signal $x$ known to lie in a simple closed convex set $C$ (e.g., box constraints \cite{chang_phase_2016}) from observations $y = Ax + \eta$, where $\eta$ is a noise vector. For a given $\tau \ge 0$, the problem can be formulated as finding $x \in C$ such that $\norm{Ax - y} \le \tau$, or equivalently, as finding the optimal solution set of the convex composite function $\varphi$ given by
\begin{equation}\label{eq:example_phi}
\varphi(x) = \dist(Ax, \mathcal{B}(y, \tau))^2 + \delta_C(x),
\end{equation}
where $\dist(Ax, \mathcal{B}(y, \tau))^2$ is $L$-Lipschitz smooth with $L = 2\|A\|^2$ \cite{bednarczuk_forward-backward_2024}. When $\tau = 0$, the smooth part of $\varphi$ corresponds to standard least-squares; when $\tau > 0$, it corresponds to modeling bounded noise (e.g., \cite{fuchs_recovery_2005}). Since the problem is often ill-posed, a bilevel problem can be constructed with the convex composite outer function
\[
\omega(x) = \lambda_1 \norm{S_1 x}_2^2 + \lambda_2 \norm{S_2 x}_1,
\]
where $S_1 \in \mathbb{R}^{q_1\times n}$, $S_2 \in \mathbb{R}^{q_2 \times n}$, $q_1, q_2 \in \mathbb{N}$, $\lambda_1, \lambda_2 \ge 0$, and $\lambda_1 + \lambda_2 > 0$. Such functions are commonly used for regularization and include total variation \cite{rudin_nonlinear_1992}, total generalized variation \cite{bredies_total_2010}, and general combinations of $\ell_1$ and $\ell_2$ norms \cite{huang_traction_2019}. While in the example above the function $g_1(x)=\lambda_2\norm{S_2x}_1$ may not be proximal friendly for certain choices of $S_2$, we show that the convergence results established in this work remain valid through the use of suitable surrogate functions  (see Remark~\ref{rem:lifting}).

\subsection{Composite Functions and the Proximal Gradient Methods} 
Throughout our analysis we use the  following properties
concerning smooth functions and the proximal operator, which we give here for the sake of completeness.%

The first result is the well-known descent lemma, which provides a quadratic upper bound for smooth functions \cite[Lemma 5.7]{beck2017}. The second result, called the second prox theorem {\cite[Theorem 6.39]{beck2017}}, provides a property of the proximal operator.
\begin{lemma}[Descent Lemma]\label{descent lemma}
Let $f:D\to \R$ be an $L_f$-smooth function over a convex set $D$. Then for any $x,y\in D$, and $L\geq L_f$
\begin{equation*}
f(y)\leq f(x)+\inner{\nabla f(x)}{y-x} +\frac{L}{2}\norm{x-y}^2.
\end{equation*}
\end{lemma}

\begin{lemma}[Second Prox Theorem]\label{second prox theorem}
  Let $g:\R^n\to \bar{\R}$ be a proper, convex, and lower semicontinuous function. Then for $x,u\in \R^n$, $u = \prox{g}{x}$ if and only if 
  $$\langle x - u, y - u\rangle \leq g(y) - g(u),\quad\forall y\in \R^n.$$ 
\end{lemma}
 
Moreover, our convergence results rely on the following property of the proximal operator of the sum $g_2+\sigma g_1$ as $\sigma$ approaches $0$ from above.
\begin{lemma}\label{lem:prox_convergence}
Let $g_1:\R^n\rightarrow\bar{\R}$ and $g_2:\R^n\rightarrow\bar{\R}$ be proper, convex, and closed functions, with $\dom(g_1)$ being a closed set. Then, for any $x\in\R^n$ 
$$\lim_{\sigma\downarrow 0}\prox{g_2+\sigma g_1}{x}=\prox{g_2+\delta_{\dom(g_1)}}{x}$$
\end{lemma}
\begin{proof}
Let $x\in \R^n$. We first note that since $\dom(g_1)$ is nonempty, closed, and convex, the function $\delta_{\dom(g_1)}$ is proper, closed, and convex. Define,
    $y({\sigma})=\prox{g_2+\sigma g_1}{x}$ and $
    w=\prox{g_2+\delta_{\dom(g_1)}}{x}$.
By \cite[Theorem 6.39]{beck2017} we have that $w,y(\sigma)\in\dom(g_1)$ and
\begin{align*}
    \inner{x-y({\sigma})}{w-y({\sigma})}&\leq g_2(w)+\sigma g_1(w)-g_2(y({\sigma}))-\sigma g_1(y({\sigma}))\\
    \inner{x-w}{y({\sigma})-w}&\leq g_2(y({\sigma}))-g_2(w).
\end{align*}
Summing these two inequalities results in
\begin{align}
\norm{w-y(\sigma)}^2\leq \sigma(g_1(w)-g_1(y(\sigma))). \label{eq:bound_dist_prox}
\end{align}
Using similar arguments, we obtain that for any $\sigma_1\geq\sigma_2>0$ the following holds:
$$\norm{y(\sigma_2)-y(\sigma_1)}^2\leq (\sigma_1-\sigma_2)(g_1(y(\sigma_2))-g_1(y(\sigma_1))).$$
Thus, we can conclude that 
$g_1(y(\sigma))$ is monotonically decreasing in $\sigma\in\R_{++}$.
Therefore, \eqref{eq:bound_dist_prox} implies that
\begin{align*}
\lim_{\sigma\downarrow 0}\norm{w-y(\sigma)}^2\leq \lim_{\sigma\downarrow 0}\sigma(g_1(w)-g_1(y(\sigma)))\leq \lim_{\sigma\downarrow 0}\sigma(g_1(w)-g_1(y(1)))=0,
\end{align*}
where the last equality follows from $w\in\dom(g_1)$, and the claim is established. %
\end{proof}

The methods analyzed in this work use the PG operator on a convex composite function $\phi=f+g$ (as in Definition~\ref{def:convex_composite}).
The PG operator of function $\phi$ for $t>0$ is given by
\begin{equation*}\label{eq:prox_grad}T^f_{g,t}(x)=\prox{t g}{x-t \nabla f(x)}.\end{equation*}
The following result provides useful bounds on the value of any solution $z$ compared to $x^+=T^f_{g,t}(x)$. This result is given by a direct application of \cite[Theorem 10.16]{beck2017} to a convex function $f$.
\begin{lemma}\label{lem:genprox}
     Let $\phi=f+g$ be a convex composite function as in Definition \ref{def:convex_composite}, and let %
      $x^+ = T^f_{g,t}(x)$ %
       where $t>0$ satisfies
     \begin{equation}\label{eq:suff_decrease}f(x^+)\leq f(x)+\inner{\nabla f(x)}{x^+-x}+\frac{1}{2t}\norm{x^+-x}^2.
     \end{equation}
     Then for any $z\in \R^n$,
     \begin{equation*}
      \phi(x^+) - \phi(z)\leq \frac{1}{2t}(\|z - x\|^2-\|z - x^+\|^2).%
     \end{equation*}
 \end{lemma}
 In the well-known PG algorithm \cite[Page 271]{beck2017}, the PG step is iteratively applied so that
 $x^{k+1}=T^f_{g,t_k}(x^k)$
 with $t_k$ satisfying \eqref{eq:suff_decrease} with respect to $x^k$. The parameter $t_k$ used in the PG step is referred to as the step size and may change between iterations.
 When $f$ is $L_f$-smooth, setting $t_k={1}/{L_f}$ (constant step size) satisfies the condition in \eqref{eq:suff_decrease}. However, when $L_f$ is unknown or hard to compute, or when local properties of the function enable the use of larger step sizes, the backtracking procedure, specified in Algorithm~\ref{alg:backtracking}, can be applied to ensure $t$ satisfies \eqref{eq:suff_decrease}. Note that the step size resulting from the backtracking procedure satisfies $t\geq \min\{{\gamma}/{L_f},\bar{t}\}$ and $t\leq \bar{t}$. This backtracking procedure may require additional (possibly costly) function evaluations, and thus trades off the opportunity to reduce the number of iterations by using larger step sizes with increasing the per-iteration computational cost.
 \begin{varalgorithm}{Backtracking}
  \caption{Backtracking procedure $\mathcal{B}^f_g(x,{\bar{t}})$}\label{alg:backtracking}
  \noindent{\bf Parameters:}  $\gamma\in(0,1)$.\\
  \noindent{\bf} Find 
  $$i^*=\min\left\{i\in\N_0: \begin{array}{l}
  t=\bar{t}\gamma^i, x^+=T^f_{g,t}(x),\\
  f(x^+)\leq f(x)+\inner{\nabla f(x)}{x^+-x}+\frac{1}{2t}\norm{x^+-x}^2
  \end{array}
 \right\}$$
 and return $t=\bar{t}\gamma^{i^*}$.
 \end{varalgorithm} 

 Applying the PG method to a convex composite function $\phi=f+g$, with minimum $\phi^*$, either with constant or backtracking step size %
 results in a convergence of the optimality gap $\phi(x^k)-\phi^*$ at a rate of $\mathcal{O}(1/k)$ \cite[Section 10.4]{beck2017}. 
 In Section~\ref{sec:four}, we will use the above results to analyze the PG-based algorithm for solving problem~\eqref{Bilevel}.

The APG method, also known as the fast iterative shrinkage-thresholding algorithm (FISTA) \cite{becteb2009} can also be used for minimization of convex composite function $\phi$. This algorithm uses two sequences $\{x^k\}_{k\in\N}$ and $\{y^k\}_{k\in\N}$ initialized so that $y^0=x^0$, and
updating steps given by
\begin{align}
            x^{k} &= T^f_{g,t_{k}}(y^{k-1})\equiv\prox{t_{k}g}{y^{k-1} - t_{k}\nabla f(y^{k-1})},\label{fistaxk}\\
            y^{k} &= x^{k} + (\frac{s_{k-1}-1}{s_{k}})(x^{k}-x^{k-1}),\nonumber
\end{align}
where $s_0=1$ and the sequence $\{s_k\}_{k\in\N}$ satisfies  {$s_{k} \leq  (1+\sqrt{1+4s^2_{k-1}})/{2}$ and $s_k\geq (k+2)/2$}. The APG step-size rule is similar to that of PG: the step size $t_k$ can be chosen either as the constant  ${1}/{L_f}$ or via the backtracking procedure in Algorithm~\ref{alg:backtracking}.
The APG method accelerates the decrease in the optimality gap of the PG method from the rate $\mathcal{O}(1/k)$ to $\mathcal{O}(1/k^2)$. The following two results concerning the sequences generated by APG are important to our convergence analysis of the APG-based algorithm, which is detailed in Section~\ref{sec:five}.
\begin{lemma}\label{lem:property_fista_seq}
   Let $\phi = f+g$ be a convex composite function as in Definition \ref{def:convex_composite}, and let $\{x^k\}_{k\in\N}$ be the sequence generated by APG where $t_k$ either is chosen as ${1}/{L_f}$ or is obtained by Algorithm \ref{alg:backtracking}. Then for any $x\in \R^n$ and $k\in\N$
   it holds that
    \begin{equation}
       t_k(\phi(x) - \phi(x^{k}))
       \geq\frac{1}{2}(\norm{x^k-x}^2-\norm{y^{k-1}-x}^2).\label{2.18a}
   \end{equation}
   \end{lemma}
   \begin{proof}
   Using Lemma \ref{second prox theorem}, we have from \eqref{fistaxk} that
   \begin{eqnarray}
       \langle y^{k-1} - x^{k}, x - x^{k}\rangle
       &=&\langle y^{k-1} - t_{k}\nabla f(y^{k-1})
       - x^{k}, x - x^{k}\rangle  + t_{k}\langle \nabla f(y^{k-1}), x - x^{k}\rangle\nonumber\\
       &\leq& t_{k}(g(x) - g(x^{k})) + t_{k}\langle \nabla f(y^{k-1}), x - x^{k}\rangle.\label{2.17} 
   \end{eqnarray}
   Note that 
   \begin{eqnarray}
       \langle \nabla f(y^{k-1}), x - x^{k}\rangle&=& \langle \nabla f(y^{k-1}), x - y^{k-1}\rangle + \langle \nabla f(y^{k-1}), y^{k-1} - x^{k}\rangle\nonumber\\
       &\leq& f(x) - f(y^{k-1}) + f(y^{k-1})- f(x^{k}) + \frac{1}{2t_{k}}\|y^{k-1} - x^{k}\|^2\nonumber\\
       &=& f(x) - f(x^{k}) + \frac{1}{2t_{k}}\|y^{k-1} - x^{k}\|^2,\label{this}
   \end{eqnarray}
   where the above inequality follows from the convexity of $f$ and the choice of $t_k$ and Lemma~\ref{descent lemma}.
   Plugging \eqref{this} into \eqref{2.17} and using
   $$\langle y^{k-1} - x^{k}, x - x^{k}\rangle = \frac{1}{2}\|y^{k-1} - x^{k}\|^2+\frac{1}{2}\|x - x^k\|^2-\frac{1}{2}\|y^{k-1} - x\|^2,$$
   we obtain the desired result.
       \end{proof}

       \begin{lemma}\label{lem:seq}
       Let $\{s_k\}_{k\in \mathbb{N}_0}$ be the sequence satisfying $s_{k+1}= \frac{1+\sqrt{1+4s^2_{k}}}{2}$ and $s_0=1$.
        Then $\frac{k+2}{2}\leq s_k\leq k+1$ for all $k\in \mathbb{N}_0$.
   \end{lemma}
   \begin{proof}
       The part $\frac{k+2}{2}\leq s_k$ is \cite[Lemma 10.33]{beck2017}. To prove the part $s_k\leq k+1$ we use induction. Since $s_0 \leq 1$, the claim trivially holds for $k=0$.  Assuming the claim holds for $k\leq n$, we prove that it also holds for $k=n+1$. By definition of $s_{n+1}$,
       \begin{equation}
           s_{n+1} = \frac{1+\sqrt{1+4s^2_n}}{2}\leq %
           \frac{1+\sqrt{1+4(n+1)^2}}{2} %
           \leq\frac{1+\sqrt{(2n+3)^2}}{2}%
           =n+2,\label{seq}
       \end{equation}
       where the first inequality follows from the induction assumption that $s_n\leq n+1$. Thus, by induction, we obtain from \eqref{seq} that $s_k\leq k+1$ for all $k\in \mathbb{N}_0$. Therefore, $\frac{k+2}{2}\leq s_k\leq k+1$, for $s_0=1$ and all $k\in \mathbb{N}$.
       \end{proof}

We conclude this section with a technical lemma, whose proof is relegated to the Appendix.
\begin{lemma}\label{lem:sum_bounds}
    Let $\beta\in\R_+$. Then for any $K\in\N$
    \begin{enumerate}[label=(\roman*)]
        \item \; \vspace{-1.5em}\begin{equation*}\sum_{k=1}^{K}k^{1-\beta}\leq \begin{cases}
\frac{(K+1)^{2-\beta}}{2-\beta}, & \beta\in[0,2),\\
(1 + \log (K)), & \beta=2,\\
\frac{\beta - 1}{\beta - 2}, & \beta>2,
\end{cases}\end{equation*}
\item \; \vspace{-1.5em}\begin{equation*}\sum_{k=1}^{K}{k^{1-\beta}}\geq\begin{cases}
       {\frac{K^{2-\beta}}{2}}, &{\beta\in[0,2)},\\
        \log (K+1), & \beta=2,%
   \end{cases}\end{equation*}
\item\; \vspace{-1.5em}\begin{equation*}\sum_{k=1}^{K-1}\frac{2\log(k)}{(k+1)^3} +\frac{\log(K)}{(K+1)^2}\leq 1.\end{equation*}
\end{enumerate}
\end{lemma}       
\section{Rates for Bilevel Problem with Composite Functions}\label{sec:rate_translate}
Our \eqref{Bilevel} model assumes that both $\varphi$ and $\omega$ are composite functions.
Our methods assume that a PG step can be computed on the function
$$\phi_k(x)\equiv\varphi(x)+\sigma_k\omega(x).$$
The function $\phi_k$ is indeed a convex composite function, since it can be written as
$$\phi_k(x)=F_k(x)+G_k(x), $$
where $F_k=f_2+\sigma_k f_1$ is Lipschitz smooth with Lipschitz constant $\sigma_kL_1+L_2$, and $G_k=g_2+\sigma_kg_1$ is a proper, convex, and lower semicontinuous function. However, computing a proximal operator on function $G_k$ may be computationally challenging even when computing it separately on  $g_1$ and $g_2$ is easy \cite{adly2019decomposition}. Therefore, in this paper we adopt the approach suggested in \cite{doronshtern2023} for the case where $G_k$ is not proximal friendly, which allows us to decompose the proximal operators of $g_1$ and $g_2$. To this end we define two surrogate functions defined on the point $w=(x,p)\in\R^{2n}$:
\begin{equation*}
\tilde{\varphi}(w)\equiv\varphi(x)+\frac{\rho}{2}\norm{x-p}^2, ~~
\tilde{\omega}(w)\equiv\omega(p).
\end{equation*}
These two functions are also convex and composite functions according to Definition~\ref{def:convex_composite} and thus satisfy Assumption~\ref{ass:composite}.
Moreover, the optimal value of $\tilde{\varphi}$ is $\varphi^*$, and is obtained when $p=x$. Thus, the problem
\begin{equation}\label{eq:BLO_surrogate}
    \min_{w=(x,p)\in\tilde{X}^*} \tilde{\omega}(w)
\end{equation}
for $\tilde{X}^*=\argmin_{\tilde{w}\in\R^{2n}}\tilde\varphi(\tilde{w})$
has the same optimal value and optimal solution set with respect to variable $x$ as problem \eqref{Bilevel}. Therefore, we can conclude that
\begin{equation*}
\min_{w\in\R^{2n}}\tilde{\omega}(w)=\ubar{\omega}, ~~
\min_{w\in\tilde{X}^*}\tilde{\omega}(w)=\min_{x\in X^*}\omega(x)=\omega^*.
\end{equation*}
Furthermore, when $\argmin_{p\in X^*} \omega(p)$ is bounded, and since $\tilde{X}^*\subseteq\{(x,p)\in\R^{2n}:x=p\}$ the set $\argmin_{w\in \tilde{X}^*} \tilde{\omega}(w)$ is also bounded. Thus, $\tilde{\omega}$ together with $\tilde{\varphi}$ satisfies Assumption~\ref{ass:omega_min}.
We note that while this decomposition approach has also been discussed in \cite[Remark 2.1]{latafat2024},  its implications to convergence rates with respect to $\omega$ and $\varphi$ were not analyzed.

The benefit of using these surrogate functions comes into play when computing a PG step on the new convex and composite function
$$\tilde{\phi}_k(w)=\tilde{F}_k(w)+\tilde{G}_k(w),$$
where
\begin{align*}
\tilde{F}_k(w)&=f_2(x)+\frac{\rho}{2}\norm{x-p}^2+\sigma_k  f_1(p),\quad
\tilde{G}_k(w)=g_2(x)+\sigma_k g_1(p).
\end{align*}
Applying the PG operator to function $\tilde{\phi}_k(w)$ at point $w=(x,p)$ is simply obtained by
\begin{equation*}T^{\tilde{F}_k}_{\tilde{G}_k,t}(w)=\prox{t\tilde{G}_k}{w-t\nabla \tilde{F}_k(w)}=\begin{bmatrix} \prox{tg_2}{x-t(\nabla f_2(x)+\rho(x-p))}\\
\prox{t\sigma_kg_1}{p-t(\rho(p-x)+\sigma_k\nabla f_1(p))}
\end{bmatrix},\end{equation*}
which involves computing separate proximal operators on $g_1$ and $g_2$. Note that the operator $T^{\tilde{F}_k}_{\tilde{G}_k,t}(w)$ depends on the choice of $\rho$.

To address the convergence rate guarantees that can be obtained for the original functions given 
a rate on the surrogate functions, we provide the following result.
\begin{prop}\label{prop:interpert_rate}
  Let $\mathcal{A}$ be an algorithm to solve problem~\eqref{Bilevel} with outer function $\tilde{\omega}$ and inner function $\tilde{\varphi}$ that generates a sequence $\{w^k\equiv(x^k,p^k)\}_{k\in\N}$ such that 
\begin{align*}
\tilde{\varphi}(w^k)-\varphi^*\leq \theta_{\varphi}(k),\quad\text{ and }\quad
\tilde{\omega}(w^k)-\omega^*\leq \theta_{\omega}(k),
\end{align*}
where $\theta_{\varphi}(k)$ and $\theta_{\omega}(k)$ are nonnegative decreasing functions of $k$ which go to zero as $k\rightarrow \infty$. Then,  
\begin{enumerate}[label=(\roman*)]
\item \label{prop:interpert_rate_i} For all $k\in \N$
\begin{equation}
    \varphi(x^k)-\varphi^*\leq \theta_{\varphi}(k),\;\omega(p^k)-\omega^*\leq \theta_{\omega}(k),\;\text{ and }\;
    \norm{x^k-p^k}\leq \sqrt{\frac{2}{\rho}\theta_{\varphi}(k)}.\label{eq:translate_conv_guarantees}
\end{equation}
\item\label{prop:interpert_rate_ii} For all $k\in \N$, 
$p^k\in S_p\equiv\Lev_{\omega}(\omega^*+\theta_\omega(1))$ and
$x^k\in S_x$  where $S_x\equiv S_p+\mathcal{B}[0,\sqrt{\frac{2}{\rho}\theta_{\varphi}(1)}]$.
\item\label{prop:interpert_rate_iii} Assume that $\omega$ is also coercive. If $g_1$ is a real-valued function, then  
\begin{equation*}
    \omega(x^k)-\omega^*\leq \theta_{\omega}(k)+\ell_1\sqrt{\frac{2}{\rho}\theta_{\varphi}(k)},
\end{equation*}
where $\ell_1$ is the Lipschitz constant of $\omega$ over the compact set $S_x$.
Alternatively, if $g_2$ is a real-valued function,  
\begin{equation*}
    \varphi(p^k)-\varphi^*\leq \theta_{\varphi}(k)+\ell_2\sqrt{\frac{2}{\rho}\theta_{\varphi}(k)}.
\end{equation*}
where $\ell_2$ is the Lipschitz constant of $\varphi$ over $S_x$.
\end{enumerate}
\end{prop}
\begin{proof}
\begin{enumerate}[label=(\roman*)]
\item The inequalities are a direct consequence of $\tilde{\omega}(w^k)=\omega(p^k)$ and 
$$\tilde{\varphi}(w^k)-\varphi^*=(\varphi(x^k)-\varphi^*)+\frac{\rho}{2}\norm{x^k-p^k}^2$$
with both $(\varphi(x^k)-\varphi^*)$ and $\frac{\rho}{2}\norm{x^k-p^k}^2$ being nonnegative.
\item It  follows from the guarantees on $\omega(p^k)$ in (i) and $\theta_{\omega}(k)\leq \theta_{\omega}(1)$  that $\{p^k\}_{k\in\N}$ is a subset of $S_p\equiv\Lev_{\omega}(\omega^*+\theta_\omega(1))$. Thus, from the guarantees on $\norm{x^k-p^k}$ in (i) and $\theta_{\varphi}(k)\leq \theta_{\varphi}(1)$,  we obtain that $\{x^k\}_{k\in\N}\subseteq S_x$, where $S_x\equiv S_p+B[0,\sqrt{\frac{2}{\rho}\theta_{\varphi}(1)}]$.
\item 
Since $\omega$ is proper, convex, lower semicontinuous, and coercive, the sets $S_p$ and $S_x$ are
convex and compact.
If $g_1$ is real-valued,  then from {the} convexity of $g_1$, $g_1$ and $\omega$ are continuous over $\R^n$ \cite[Theorem 2.21]{beck2017}.
Thus, $\omega$ is Lipschitz continuous over $S_x$, with $\ell_1$ being its Lipschitz constant. We conclude that
$$\hspace{-5pt}\omega(x^k)-\omega^*\hspace{-4pt}=\hspace{-2pt}\omega(x^k)-\omega(p^k)+\omega(p^k)-\omega^*\hspace{-4pt}\leq \hspace{-2pt}\ell_1\norm{x^k-p^k}+\theta_{\omega}(k)\hspace{-2pt}\leq\hspace{-2pt}\ell_1\sqrt{\frac{2\theta_{\varphi}(k)}{\rho}}+\theta_{\omega}(k),$$
where the inequalities follow from the Lipschitz continuity of $\omega$ and \eqref{eq:translate_conv_guarantees}.\\
If $g_2$ is real-valued then applying the same logic, $\varphi$ is Lipschitz continuous over $S_x$ with constant $\ell_2$, and
$$\hspace{-5pt}\varphi(p^k)-\varphi^*\hspace{-4pt}=\hspace{-2pt}\varphi(p^k)-\varphi(x^k)+\varphi(x^k)-\varphi^*\hspace{-4pt}\leq \hspace{-2pt}\ell_2\norm{x^k-p^k}+\theta_{\varphi}(k)\hspace{-2pt}\leq \hspace{-2pt}\ell_2\sqrt{\frac{2\theta_{\varphi}(k)}{\rho}}+\theta_{\varphi}(k).$$
\end{enumerate}
\end{proof}
Thus, Proposition \ref{prop:interpert_rate}\ref{prop:interpert_rate_i} provides rate of convergence guarantees for the case where both $\omega$ and $\varphi$ are extended real-valued functions. Since at any point $w^k=(x^k,p^k)$ we might have that $x^k\notin \dom(\omega)$ and $p^k\notin\dom(\varphi)$, it provides a rate at which $\dist(x^k,\dom(\omega))$ and $\dist(p^k,\dom(\varphi))$ converge to $0$. 
Moreover, Proposition~\ref{prop:interpert_rate}\ref{prop:interpert_rate_iii} addresses the case where one of the functions is real-valued and translates the convergence rates for the surrogate functions' sequences $\tilde{\varphi}(w^k)$ and $\tilde{\omega}(w^k)$ directly to a convergence rate for the values of the original functions, $\varphi$ and $\omega$ for either $\{x^k\}_{k\in\N}$ or $\{p^k\}_{k\in\N}$.

\begin{remark}\label{rem:lifting}
A similar decomposition technique may also prove useful when the outer function is of the form $\omega(Sx)$ for some matrix $S\in\R^{\ell\times n}$, for example, $\omega(Sx)=\norm{Sx}_1$. This allows us to use the PG step on $\omega(p)$ with respect to $p$ in situations where the PG step for $\omega(Sx)$ with respect to $x$ cannot be easily computed. Indeed, defining
$$\tilde{\varphi}(w)=\varphi(x)+\frac{\rho}{2}\norm{Sx-p}^2,$$
an analogue of Proposition \ref{prop:interpert_rate}\ref{prop:interpert_rate_iii} for $g_1$ being a real-valued function guarantees that
\begin{equation*}\varphi(x^k) - \varphi^*\leq \theta_\varphi(k),\quad\text{ and }\quad
\omega(Sx^k)-\omega^*\leq \theta_\omega(k)+\ell_1\sqrt{\frac{2}{\rho}\theta_\varphi(k)}.
\end{equation*}
\end{remark}
\begin{remark}\label{rem:rho_influence}
The use of the surrogate functions is valid for any $\rho>0$, but this parameter affects the Lipschitz constant of the smooth part of $\tilde{\varphi}$, increasing it to $\tilde{L}_2 = L_2 + 2\rho$. The bounds $\theta_{\varphi}$ and $\theta_{\omega}$ are typically monotonically increasing in $\tilde{L}_2$ and in the case of our analysis in Sections~\ref{sec:four} and \ref{sec:five} are affine functions of $\tilde{L}_2$. Assume that $\theta_{\tau}(k)=\theta_{\tau}^0(k)+\theta_{\tau}^{\rho}(k) \rho$ for $\tau\in\{\varphi,\omega\}$, where $\theta_{\tau}^0(k), \theta_{\tau}^{\rho}(k)$ are nonnegative decreasing and converging to $0$ with $k\rightarrow\infty$. Thus,  Proposition~\ref{prop:interpert_rate}\ref{prop:interpert_rate_i}  yields
$$\norm{x^k-p^k}\leq \sqrt{\frac{2\theta_{\varphi}^0(k)}{\rho}+2\theta_{\varphi}^{\rho}(k)}.$$
Therefore, the distance between $x^k$ and $p^k$ decreases with the increase in $\rho$, creating  a trade-off between the bound on the convergence of the function values and the bound on this distance. This trade-off is even more evident in Proposition~\ref{prop:interpert_rate}\ref{prop:interpert_rate_iii}, where, depending on the specific forms of $\theta_\omega$ and $\theta_{\varphi}$, there may exist a value $\rho \in (0,\infty)$ that minimizes the bound for the real-valued function.
\end{remark}
In the following sections, we present two algorithms to solve the bilevel problem \eqref{Bilevel} based on applying the PG operator on function $\phi_k$ and derive convergence rate guarantees for these algorithms. To address the cases where these algorithms cannot directly be applied to the bilevel problem~\eqref{Bilevel}, due to the computational burden of computing the PG operator of $\phi_k$, we provide a convergence rate assuming that the algorithm is applied on the surrogate problem~\eqref{eq:BLO_surrogate}.

\section{Iteratively REgularized Proximal Gradient}\label{sec:four}
In this section, we present a PG-based algorithm for solving problem \eqref{Bilevel}. This iterative algorithm applies a PG step on the iterative regularized function $\phi_k$ defined in \eqref{Tikhonov}. The full description of the method is recorded in Algorithm~\ref{alg:IRE-PG}.

\begin{varalgorithm}{IRE-PG}
    \caption{
    Iteratively REgularized Proximal Gradient}\label{alg:IRE-PG}
        \noindent {\bf Input}: A sequence $\{\sigma_k\}_{k\in\N}\subset\R_+$, scalar $\bar{t}>0$, %
        and an initial point $x^0\in \R^n$.\\
        \noindent {\bf General Step}: For $k=1,2,\dots$
        \begin{enumerate}
            \item[1.] Set $F_k=\sigma_kf_1 + f_2$, and $G_k=\sigma_kg_1 + g_2$. %
            \item[2.] Choose $t_k$ from these options
            \begin{enumerate}[label={(\alph*)}]
                \item $t_k=\frac{1}{L_2+\sigma_kL_1}$ (constant step size).
                \item $t_k=\mathcal{B}^{F_k}_{G_k}(x^{k-1},{\bar{t}})$ defined by Algorithm~\ref{alg:backtracking} (backtracking step size).
            \end{enumerate}
            \item[3.]Compute
            $x^{k} = T^{F_k}_{G_k,t_k}(x^{k-1})\equiv \prox{t_kG_k}{x^{k-1} - t_k\nabla F_k(x^{k-1})}.$%
        \end{enumerate}
\end{varalgorithm}

 If $g_1=0$ and $g_2 = \delta_C$, where $C$ is a nonempty, closed, and convex subset of $\R^n$, then \eqref{Bilevel} reduces to the bilevel problem studied in \cite{Sol}. Thus, 
 the explicit descent method \cite[Algorithm 2]{Sol} is a special case of Algorithm~\ref{alg:IRE-PG}.    

In the following {sub}section, we prove that the sequence generated by \ref{alg:IRE-PG} is bounded and all its limit points are solutions to problem~\eqref{Bilevel}. 

\subsection{Convergence analysis of IRE-PG}\label{sec:IRE_PG_converge}
Our convergence proof is heavily inspired by the convergence proof in \cite{Sol}, with the appropriate modifications for the proximal setting.
In order to prove the properties of the sequence $\{x^k\}_{k\in\N}$ produced by \ref{alg:IRE-PG}, we first give the following auxiliary result. 
\begin{lemma}\label{lem:IRE-PG}
    Let $x^*\in \ubar{X}$,   
    and let $\{x^k\}_{k\in\N_0}$ be the sequence generated by \ref{alg:IRE-PG}.
    Then for all $k\in\N$
    $$\|x^{k} - x^*\|^2\leq \|x^{k-1} - x^*\|^2 + 2\sigma_kt_k(\omega^* - \omega(x^{k})) +
    2t_k(\varphi^* - \varphi(x^{k})).$$  
\end{lemma}
\begin{proof}
 The result follows from Lemma 
 \ref{lem:genprox} by replacing $f$ and $g$ with $F_k$ and $G_k$, respectively, and setting $z:=x^*$, $x:=x^{k-1}$, and $x^+:=x^{k}$. 
\end{proof}

To obtain the convergence results in this section, we henceforth apply the following assumption on the sequence of regularizing parameters $\{\sigma_k\}_{k\in\N}$.
\begin{ass}\label{ass:sigma}
    The sequence $\{\sigma_k\}_{k\in\N}\subseteq \R_{+}$ is nonincreasing, $\lim_{k\rightarrow \infty} \sigma_k=0$ and $\sum_{k\in\N} \sigma_k=\infty$.
\end{ass}
With this assumption, we prove that the sequence generated by \ref{alg:IRE-PG} is bounded.
\begin{lemma}\label{lem:bounded}
   Let $x^*\in \ubar{X}$, and  
    let $\{x^k\}_{k\in \mathbb{N}_0}$ be a sequence generated by \ref{alg:IRE-PG}.
    Then, there exists a constant $D$ such that $\norm{x^k-x^*}\leq D$ for all $k\in\N_0$.
\end{lemma}
\begin{proof}
We start by showing that the sequence of optimality gaps $\varphi(x^k)-\varphi^*$ is bounded for all $k\in\N$.
 Replacing $f$ and $g$ with $F_k$ and $G_k$, respectively, in Lemma \ref{lem:genprox}, and setting $z=x=x^{k-1}$ and $x^+=x^{k}$, results in 
\begin{equation}\label{eq:decrease}\phi_k(x^{k})\leq \phi_k(x^{k-1})-\frac{1}{2t_k}\norm{x^k-x^{k-1}}^2.\end{equation}
Moreover, since $\sigma_k\geq \sigma_{k+1}\geq 0$, $\omega(x^k)\geq \ubar{\omega}$, and using \eqref{eq:decrease}  we obtain that
\begin{align}
    (\varphi(x^{k})-\varphi^*)+\sigma_{k+1}(\omega(x^{k})-\ubar{\omega})
    &\leq (\varphi(x^{k})-\varphi^*)+\sigma_{k}(\omega(x^{k})-\ubar{\omega})\nonumber\\
&= \phi_{k}(x^{k})-\varphi^*-\sigma_{k}\ubar{\omega}\nonumber\\
&\leq\phi_{k}(x^{k-1})-\varphi^*-\sigma_{k}\ubar{\omega}-\frac{1}{2t_k}\norm{x^k-x^{k-1}}^2\nonumber\\
&\leq  \varphi(x^{k-1})-\varphi^*+\sigma_{k}(\omega(x^{k-1})-\ubar{\omega}).\label{eq:gap_decrease}
\end{align}
Applying \eqref{eq:gap_decrease} recursively, and using the fact that $\varphi(x^{k})-\varphi^*\geq 0$ and $\omega(x^{k})-\ubar{\omega}\geq 0$, we obtain that
\begin{align*}
\varphi(x^{k})-\varphi^*&\leq (\varphi(x^{k})-\varphi^*)+\sigma_{k+1}(\omega(x^{k})-\ubar{\omega})
\leq (\varphi(x^{0})-\varphi^*)+\sigma_{1}(\omega(x^{0})-\ubar{\omega})\equiv \nu.
\end{align*}
Consider the function $\Psi:\R^n\rightarrow\bar{\R}$ defined as
\begin{equation*}\Psi(x)=\max\{\varphi(x)-\varphi^*,\omega(x)-\omega^*\}.\end{equation*}
By Assumption~\ref{ass:composite} and due to $\dom(g_1)\cap\dom(g_2)\neq \emptyset$, we conclude that
$\Psi$ is proper, convex, and lower semicontinuous.
Moreover, from Assumption~\ref{ass:omega_min}, we conclude that its minimum is zero, and $\Lev_{\Psi}(0)$ is bounded since
$$\Lev\nolimits_{\Psi}(0)=\argmin_{x\in\R^n}\Psi(x)=\argmin_{x\in X^*}\omega(x).$$ 
Thus, by \cite[Proposition 11.3]{bauschkecombettes2017}, all of the level sets of $\Psi$ are bounded.

To show that the sequence $\{x^k\}_{k\in\N}$ is bounded, we divide the iterations into 
$\mathcal{K}=\{k\in \N:\omega(x^*)\leq\omega(x^{k})\}$ and $\bar{\mathcal{K}}=\N \setminus\mathcal{K}$.
On one hand, if $k\in \mathcal{K}$ then it follows from Lemma \ref{lem:IRE-PG} that $\|x^{k} - x^*\|\leq \|x^{k-1} - x^*\|$. On the other hand, if  $k\in\bar{\mathcal{K}}$,
then 
$$\Psi(x^{k})=\varphi(x^{k})-\varphi^*\leq \nu,$$
where $\nu$ is defined above. Consequently, $x^{k}\in\Lev_{\Psi}(\nu)$ and thus resides in a bounded set. 
Thus, we conclude that $\{x^k\}_{k\in\N}$ is bounded, so  the desired result is obtained for $D=\max\{\sup_{y\in\Lev_{\Psi}(\nu)}\norm{y-x^*},\norm{x^0-x^*}\}$.
\end{proof}
 
We are now ready to prove (subsequence) convergence of \ref{alg:IRE-PG} to a solution of problem \eqref{Bilevel}. %
\begin{theorem}\label{prop3.1}
Let $\{x^k\}_{k\in \mathbb{N}_0}$ be a sequence generated by \ref{alg:IRE-PG}, and assume additionally that $\dom(g_1)$ is closed. Then any limit point of the sequence is a solution of problem \eqref{Bilevel}.  
\end{theorem}
\begin{proof}
We first show that any limit point of the sequence $\{x^k\}_{k\in\N}$ 
belongs to $X^*$. Indeed, the boundedness of $\{x^k\}_{k\in\N}$, established in Lemma~\ref{lem:bounded}, ensures that it has a limit point. 
Using \eqref{eq:decrease} we obtain that
\begin{align}
\frac{1}{2t_k}\norm{x^{k}-x^{k-1}}^2&\leq \sigma_k(\omega(x^{k-1})-\omega(x^k))+\varphi(x^{k-1})-\varphi(x^k)\nonumber\\
&\leq \sigma_k(\omega(x^{k-1})-\ubar{\omega})-\sigma_{k+1}(\omega(x^k)-\ubar{\omega})+\varphi(x^{k-1})-\varphi(x^{k}).\label{eq:bound_norm}\end{align}
A lower bound on $t_k$ can be obtained by $$\ubar{t}\equiv\min_{k\in\N} t_k\geq \begin{cases}
\frac{1}{L_2+\sigma_1 L_1}, & t_k \text{ is a constant step size},\\
\min\{\frac{\gamma}{L_2+\sigma_1 L_1}
,\bar{t}\}, & t_k \text{  is a backtracking step size},
\end{cases}$$
and similarly, an upper bound on $t_k$ is given by $\tilde{t}=\max\{\frac{1}{L_2},\bar{t}\}$.
Thus, 
summing \eqref{eq:bound_norm} over $k\in \N$ and using $\tilde{t}$ to bound $t_k$ we obtain that
\begin{align*}
\sum_{k\in \N} \norm{x^k-x^{k-1}}^2 \leq 2\tilde{t}(\sigma_1(\omega(x^0)-\ubar{\omega})+\varphi(x^0)-\varphi^*)\end{align*}
implying $\lim_{k\rightarrow \infty}\norm{x^k-x^{k-1}}=0.$
Let $\{x^{i_k}\}_{k\in\N}$ be some subsequence of $\{x^k\}_{k\in\N}$ converging to $\bar{x}$. We will show that $\varphi(\bar{x})=\varphi^*$.
For each $k\in\N$ it holds that
\begin{align}
\norm{\bar{x}-T^{f_2}_{G_{i_k},t_{i_k}}(\bar{x})}
&\leq  \norm{\bar{x}-T^{F_{i_k}}_{G_{i_k},t_{i_k}}(\bar{x})}+
\norm{T^{F_{i_k}}_{G_{i_k},t_{i_k}}(\bar{x})-T^{f_2}_{G_{i_k},t_{i_k}}(\bar{x})}\nonumber\\
&\leq \norm{\bar{x}-T^{F_{i_k}}_{G_{i_k},t_{i_k}}(\bar{x})}+\norm{\sigma_{i_k}t_{i_k}\nabla f_1(\bar{x})}\label{eq:bound_limit1}
\end{align}
where the second inequality is due to the nonexpansiveness of the prox operator and the definition of $F_{i_k}$, as well as
\begin{align}
\norm{\bar{x}-T^{F_{i_k}}_{G_{i_k},t_{i_k}}(\bar{x})}&= 
\norm{\bar{x}-T^{F_{i_k}}_{G_{i_k},t_{i_k}}(\bar{x})-(x^{i_k-1}-x^{i_k})+(x^{i_k-1}-x^{i_k})}\nonumber\\
&\leq \norm{\bar{x}-T^{F_{i_k}}_{G_{i_k},t_{i_k}}(\bar{x})-(x^{i_k-1}-T^{F_{i_k}}_{G_{i_k},t_{i_k}}(x^{i_k-1}))}+\norm{x^{i_k-1}-x^{i_k}}\nonumber\\
&\leq (2+(L_2+\sigma_{i_k}L_1)t_{i_k})\norm{\bar{x}-x^{i_k-1}}+\norm{x^{i_k-1}-x^{i_k}},\label{eq:bound_limit2}
\end{align}
where the first inequality follows from the triangle inequality and the definition of $x^{i_k}$, and the second inequality follows from the Lipschitz continuity of the gradient mapping \cite[Lemma 10.10]{beck2017}.
Thus, combining \eqref{eq:bound_limit1} and \eqref{eq:bound_limit2} and using the fact that $t_{i_k}\leq \tilde{t}$, $\lim_{k\rightarrow \infty}x^{i_k-1}=\bar{x}$, and $\lim_{k\rightarrow \infty}\sigma_{i_k}=0$ we obtain that
$\lim_{k\rightarrow \infty}\norm{\bar{x}-T^{f_2}_{G_{i_k},t_{i_k}}(\bar{x})}=0$. %
Moreover, from the monotonicity of the gradient mapping \cite[Theorem 10.9]{beck2017} and $t_{i_k}^{-1}\leq \ubar{t}^{-1}$ the above inequality implies that
\begin{align*}
    0=\lim_{k\rightarrow \infty}\norm{\bar{x}-T^{f_2}_{G_{i_k},t_{i_k}}(\bar{x})}&\geq 
\lim_{k\rightarrow \infty}\norm{\bar{x}-T^{f_2}_{G_{i_k},\ubar{t}}(\bar{x})}\\
&=\lim_{k\rightarrow \infty}\norm{\bar{x}-\prox{\ubar{t}(g_2+\delta_{\dom(g_1)})}{\bar{x}-\ubar{t}\nabla f_2(\bar{x})}},
\end{align*}
where the last equality is due to Lemma~\ref{lem:prox_convergence}.
Thus, we can conclude that
$\bar{x}\in\argmin_{x\in\R^n} \varphi(x)+\delta_{\dom(g_1)}(x)\subseteq \argmin_{x\in\R^n} \varphi(x)=X^*$ with the inclusion stemming from the assumption that the optimal solution set of \eqref{Bilevel} is nonempty.

To prove that any limit point of $\{x^k\}_{k\in \N}$ also belongs to $\ubar{X}$, \textit{i.e.}, is a solution to \eqref{Bilevel}, we look at two cases:\\
{\bf Case I}: There  exists $k_1\in\mathbb{N}$ such that $\omega^*\leq \omega(x^{k})$ for all $k\geq k_1$. 
We claim that $\lim\inf_{k\to \infty}\omega(x^{k}) = \omega^*$. Suppose to the contrary; then there exists
$\tilde{\omega}$ such that 
 $\lim\inf_{k\to \infty}\omega(x^{k}) = \tilde{\omega}>\omega^*$. Taking $\delta = \tilde{\omega} - \omega^*>0$, then by definition there exists $k_2\geq k_1$ such that $\omega(x^{k})>\tilde{\omega} - \frac{\delta}{2}$ for all $k\geq k_2$, that is, $\omega(x^{k}) - \omega^*\geq \frac{\delta}{2}$. Turning to Lemma \ref{lem:IRE-PG}, for any optimal solution $x^*\in\ubar{X}$ we have that for all $k\geq k_2$  
\begin{eqnarray}
  \|x^{k} - x^*\|^2&\leq& \|x^{k-1} - x^*\|^2 -t_k\delta\sigma_k%
  \leq \|x^{k-1} - x^*\|^2 -{\ubar{t}\delta\sigma_k}. \nonumber%
\end{eqnarray}
Summing the above from $k_2$ we get 
$$\ubar{t}\delta\sum_{k=k_2}^{\infty}\sigma_k\leq \|x^{k_2-1} - x^*\|^2 < +\infty,$$
which contradicts our assumption that $\{\sigma_k\}_{k\in\N}$ is not summable. %
Therefore, we conclude that $\lim\inf_{k\rightarrow\infty} \omega(x^{k}) = \omega^*$. %
Let $\{x^{j_k}\}_{k\in\N}$ be a converging subsequence of $\{x^k\}_{k\in\N}$ such that $\lim_{k\to\infty}\omega(x^{j_k})=\liminf_{k\to\infty}\omega(x^k)=\omega^*$, which exists by Lemma~\ref{lem:bounded}, and let $x^\dag$ be its limit. Then, by the lower semicontinuity of $\omega$, we have that
\begin{equation*}
 \omega(x^\dagger)\leq \lim_{k\to \infty}\omega(x^{j_k}) = \omega^*.
\end{equation*}
Combining this result with $x^\dag\in X^*$, it follows that $x^\dag\in\ubar{X}$. Moreover, since $\omega(x^k)\geq\omega^*=\omega(x^\dag)$ for all $k\geq k_1$, Lemma~\ref{lem:IRE-PG} with $x^*=x^\dag$ implies that $\norm{x^k-x^\dag}\leq\norm{x^{k-1}-x^\dag}$ for all $k\geq k_1$. Thus, we conclude that $x^k$ also converges to $x^{\dag}$. Hence, $x^\dag\in\ubar{X}$ is the unique limit point of $\{x^k\}_{k\in\N}$.\\
{\bf Case II}: For each $k\in \mathbb{N}$ there is $k_1\in \mathbb{N}$ such that $k_1\geq k$ and $\omega^*>\omega(x^{k_1})$.
For each $k$, let
$j_k:=\max\{j\leq k: \omega^*>\omega(x^{j})\}.$
Then $\{j_k\}_{k\in\N}$ is nondecreasing, and $j_k\to \infty$ as $k\to \infty$.
By the definition of $j_k$, %
 for all $k\in\N$ and $j_{k}+1\leq l\leq k$ we have $\omega(x^l)\geq \omega^*$, and therefore, by Lemma \ref{lem:IRE-PG}, for any $x^*\in\ubar{X}$
\begin{equation}
\|x^{l} - x^*\|^2\leq \|x^{l-1} - x^*\|^2 + {2t_{l}\sigma_{l}}
(\omega^* - \omega(x^{l}))  \leq  \|x^{l-1} - x^*\|^2.\label{eq:no_jk}
\end{equation}
Let $\{x^{i_k}\}_{k\in\N}$ be a converging subsequence of $\{x^k\}_{k\in\N}$, and let $x^{\ddag}$ be its limit. We have established that $x^{\ddag}\in X^*$, or equivalently $\lim_{k\in\N} \dist(x^{i_k},X^*)=0$.
For every $k\in\N$, $k\geq j_k$ by definition. Therefore, either $j_k<k$ and \eqref{eq:no_jk} can be
recursively applied to $l=j_k+1,j_k+2,\ldots,k$ to obtain
$\norm{x^k-x^*}^2\leq \norm{x^{j_k}-x^*}^2$, or $k=j_k$ and $\norm{x^k-x^*}^2=\norm{x^{j_k}-x^*}^2$.
Consequently, for all $i_k$, $\dist(x^{i_k},\ubar{X})\leq \dist(x^{j_{i_k}},\ubar{X})$.
Let $\{x^{l_{j_{i_k}}}\}_{k\in\N}$ be a subsequence of $\{x^{j_{i_k}}\}_{k\in\N}$ that converges to some $\hat{x}$.  Then, by the proof above, $\hat{x}\in X^*$, and it follows from lower semicontinuity of $\omega$ and the definition of $j_k$ that
\begin{equation*}
\omega(\hat{x})\leq \liminf_{k\to \infty}\omega(x^{l_{j_{i_k}}})\leq  \limsup_{k\to \infty}\omega(x^{j_k}) \leq \omega^*,
\end{equation*}
implying that $\hat{x}\in \ubar{X}$. Thus
\begin{equation*}
\dist(x^\ddag,\ubar{X})\hspace{-2pt}=\hspace{-2pt}\lim_{k\in \N }\dist (x^{i_k},\ubar{X})\hspace{-2pt}\leq\hspace{-2pt}
\liminf_{k\in \N }\dist (x^{j_{i_k}},\ubar{X})\hspace{-2pt}\leq \hspace{-2pt}\lim_{k\in \N }\dist (x^{l_{j_{i_k}}},\ubar{X})\hspace{-2pt}=\hspace{-2pt}\dist(\hat{x},\ubar{X})\hspace{-2pt}=\hspace{-2pt}0.
\end{equation*}
\end{proof}

\subsection{Rate of Convergence of IRE-PG}\label{sec:IRE_PG_rate}
In this subsection, we present a rate of convergence result for \ref{alg:IRE-PG}. We note that our result is given in terms of a carefully constructed ergodic sequence $\{\bar{x}^k\}_{k\in\N}$, which depends on the choice of the sequence $\{\sigma_k\}_{k\in\N}$. %
Moreover, we show that the choice of $\sigma_k$ affects the trade-off between the rates of convergence for {the inner and outer functions}
with respect to the chosen ergodic sequence.
\begin{theorem}\label{them:conv_rate_IREPG}
Let $\{x^k\}_{k\in\N}$ be the sequence generated by \ref{alg:IRE-PG}  with $\sigma_k = k^{-\beta}$ for $k\in\N$ and $\beta\in (0, 1)$. Define $\pi_k = \sigma_kt_k$, and $\bar{x}^K=\sum_{k=1}^{K}\frac{\pi_k}{\sum_{l=1}^{K} \pi_l}x^{k}$. 
Then, the following inequalities hold for any $x^*\in \ubar{X}$  and any $K\in\N$:
\begin{enumerate}[label=(\roman*)]
\item\;\vspace{-1.5em}\begin{equation*}
    \omega(\bar{x}^K)-\omega^*\leq \frac{\norm{x^0-x^*}^2}{2}\frac{{\alpha_1}}{K^{1-\beta}}, 
\end{equation*}
where $\alpha_1=L_1+L_2$ if $t_k$ is chosen as a constant step size, and  $\alpha_1=\gamma^{-1}L_1+\max\{\bar{t}^{-1},\gamma^{-1}L_2\}$ if $t_k$ is chosen through Algorithm \ref{alg:backtracking}.
\item\;\vspace{-1.5em}
\begin{align*}\varphi(\bar{x}^K)-\varphi^*&\leq \frac{\norm{x^0-x^*}^2}{2}\frac{{\alpha_1}}{K^{1-\beta}}+\Delta_\omega\begin{cases}
        \frac{\alpha_2}{K^{\beta}(1-2\beta)}, &\beta\in (0,0.5),\\
        \frac{
        \alpha_2(1+\log(K))
        }{K^{0.5}}, &\beta=0.5,\\
        \frac{2\beta\alpha_2}{K^{1-\beta}(2\beta-1)}, &\beta\in (0.5,1),
        \end{cases}
    \end{align*}
    where $\alpha_1$ is as in (i) and $\alpha_2=1$ if $t_k$ is chosen as a constant step size, and $\alpha_2=\gamma^{-1}\bar{t}L_1+\max\{1,\gamma^{-1}\bar{t}L_2\}$ if $t_k$ is chosen via Algorithm \ref{alg:backtracking}.
\end{enumerate}
\end{theorem}
\begin{proof}
\begin{enumerate}[label=(\roman*)]
\item   
Since the result holds trivially if $\omega(\bar{x}^K) <\omega^*$, we shall focus on the case when $\omega(\bar{x}^K) \geq \omega^*$. In this case, since {$\varphi(x^k)-\varphi^*\geq 0$ and $\omega$ is convex,} using Lemma \ref{lem:IRE-PG}, we obtain that
    \begin{eqnarray}        \left(\sum_{l=1}^{K}2\pi_{l}\right)(\omega(\bar{x}^K) - \omega^*)&\leq&\sum_{k=1}^{K}2\pi_k(\omega(x^{k}) - \omega^*)\nonumber\\
        &\leq& \sum_{k=1}^{K}(\|x^{k-1} - x^*\|^2 - \|x^{k} - x^*\|^2)\leq\|x^0 - x^*\|^2.\label{3.7}
    \end{eqnarray}
    Observe that by definition of the step size and since $\sigma_k\leq 1$ we have
    \begin{equation*}
        t_k\geq \ubar{t}_k=
        \begin{cases}
        \frac{1}{\sigma_kL_1+L_2}, & t_k \text{ is a constant step size}\\
         \min\{\bar{t},\frac{\gamma}{\sigma_k L_1+L_2}\}, & t_k \text{ is a backtracking step size}
\end{cases}
\end{equation*} 
Furthermore,
    the derivative of the function $s\mapsto\frac{s}{sL_1+L_2}$ for $s>0$ is given by $\frac{L_2}{(sL_1+L_2)^2}>0$ and so it is increasing. Therefore, $\pi_k=t_k\sigma_k\geq \ubar{t}_k\sigma_k=\ubar{\pi}_k$,
    where $\ubar{\pi}_k$ is nonincreasing in $k$, and so $\pi_k\geq \ubar{\pi}_k\geq\ubar{\pi}_K$ for all $k\in[K]$, and
\begin{align}\frac{1}{\sum_{l=1}^K \pi_l}&\leq \frac{1}{ K\ubar{\pi}_K}\leq \frac{1}{K\sigma_K\ubar{t}_K}\nonumber\\
    &\leq \begin{cases}
        \frac{L_1+L_2}{K^{1-\beta}}, & t_k \text{ is a constant step size,}\\
        \frac{\max\{\bar{t}^{-1},\gamma^{-1}(L_1+L_2)\}}{K^{1-\beta}}, & t_k \text{ is a backtracking step size,}
    \end{cases}\label{eq:bound_on_sum_pi}
    \end{align}
    Applying \eqref{eq:bound_on_sum_pi} to \eqref{3.7}, we obtain the desired result.
\item %
   Applying Lemma \ref{lem:IRE-PG} to bound $\pi_k(\varphi(x^k)-\varphi^*)$, we get
    \begin{eqnarray}
    \left(\sum_{k=1}^K2{\pi}_k\right)(\varphi(\bar{x}^K)-\varphi^*)
    &\leq& 2\sum_{k=1}^{K}\pi_k(\varphi(x^{k})-\varphi^*)
    \nonumber\\
    &\leq& 
    \sigma_1\norm{x^0-x^*}^2+\sum_{k=1}^{K-1}(\sigma_{k+1}-\sigma_{k})\norm{x^{k}-x^*}^2\nonumber\\
    &&-\sigma_{K}\norm{x^{K}-x^*}^2+
    \sum_{k=1}^{K}2\sigma_k\pi_k(\omega^* - \omega(x^{k}))\nonumber\\
    &\leq& \norm{x^0-x^*}^2+2\Delta_\omega\sum_{k=1}^{K}\sigma_k\pi_k%
    ,\label{3.8}
    \end{eqnarray}
    where the first inequality follows from the convexity of $\varphi$, the second inequality follows from the definition of $\pi_k$, and the last inequality is due to the fact that $\sigma_k$ is decreasing in $k$, and $\omega^*-\omega(x^k)\leq \Delta_\omega$.
    Moreover, 
    defining $$\tilde{t}_k=\begin{cases}
    \frac{1}{\sigma_kL_1+L_2}, & t_k \text{ is a constant step size},\\
    \bar{t}, & t_k \text{ is a backtracking step size},
    \end{cases}$$
    we have that $\tilde{t}_k$ is nondecreasing in $k$, and $t_k\leq \tilde{t}_k$, which implies \begin{equation}
        \label{3.12}
        \sum_{k=1}^K\sigma_k\pi_k=\sum_{k=1}^K\sigma_k^2t_k\leq \tilde{t}_K\sum_{k=1}^K \sigma_k^2.
    \end{equation}
    Plugging \eqref{3.12} into \eqref{3.8} we obtain
    \begin{align}
\varphi(\bar{x}^K)-\varphi^*&\leq \frac{1}{\sum_{k=1}^K2{\pi}_k}
    \left(\norm{x^0-x^*}^2+{2\Delta_\omega\tilde{t}_K}\sum_{k=1}^K\sigma^2_k\right)\nonumber\\
    &\leq \frac{\norm{x^0-x^*}^2}{2}\left(\frac{\alpha_1}{K^{1-\beta}}\right)+\Delta_\omega\left(\frac{\alpha_2}{K^{1-\beta}}\right)\sum_{k=1}^K\sigma_k^2\label{eq:ub_RIPG1},
    \end{align}
    where the last inequality follows from the definition of $\tilde{t}_K$, \eqref{eq:bound_on_sum_pi},
    and the definition of $\alpha_1$ and $\alpha_2$.
   Bounding the sum of  $\sigma^2_k$, we obtain \begin{equation*}\sum_{k=1}^{K}{\sigma_k^2}\leq \sum_{k=1}^{K} k^{-2\beta}\leq (1+\int_{1}^{K} x^{-2\beta}dx)\leq \begin{cases}
    \frac{K^{1-2\beta}}{1-2\beta}, & \beta\in(0,0.5),\\
    1+\log(K), &\beta=0.5,\\
   \frac{2\beta}{2\beta-1}, & \beta\in(0.5,1),
    \end{cases}%
    \end{equation*}
    which plugged into \eqref{eq:ub_RIPG1}, results in the desired expression. 
\end{enumerate}
\end{proof}
\remark\label{rem:thm_IRE_conv}{Theorem~\ref{them:conv_rate_IREPG} does not require the boundedness {of $\ubar{X}$} defined in Assumption~\ref{ass:omega_min}. Thus, the performance of the ergodic sequence is guaranteed even when we cannot prove the limit point convergence as in Theorem~\ref{prop3.1}.
Moreover, revisiting the function $\Psi(x)=\max\{\varphi(x)-\varphi^*,\omega(x)-\omega^*\}\geq 0$ defined in the proof of Lemma~\ref{lem:bounded}, we recall that it is proper, closed, and convex, and additionally, under Assumption~\ref{ass:omega_min} it has bounded level sets. 
Thus, Theorem~\ref{them:conv_rate_IREPG} implies that $\{\bar{x}^k\}_{k\in\N}$ resides in some compact level set of $\Psi$, and thus has limit points which also belong to this set. Specifically, any limit point $x^{\dag}$ of the sequence satisfies $\Psi(x^{\dag})=0$ and is thus a solution to the bilevel problem.
}

We note that the rate results obtained in Theorem~\ref{them:conv_rate_IREPG} imply that for every $\beta\in(0,0.5)$ the convergence rates of the inner and outer functions are given by $\mathcal{O}(k^{-\beta})$ and $\mathcal{O}(k^{{-(1-\beta)}})$, respectively, that is, the outer function converges faster than the inner function. Notice that this is complementary to the rate result obtained by Bi-SG \cite{merchavshoham2023}, which shows that for any $\beta\in(0,0.5)$ the rates for the inner and outer functions are given by $\mathcal{O}(k^{-(1-\beta)})$ and $\mathcal{O}(k^{-\beta})$, respectively.
This result may also be due to how the convergence is measured. Indeed, while we measure the convergence with respect to the ergodic sequence, in \cite{merchavshoham2023} the convergence is measured with respect to the $x^{K^*}$ such that
$K^*\in \argmin_{k\in\{K+1,K+2,\ldots 2K\}} \omega(x^k)$. Therefore, we also present the analysis for a similar convergence measure.
\begin{theorem}\label{thm:conv_rate_IREPG2}
    Let $\{x^k\}_{k\in\N}$ be the sequence generated by \ref{alg:IRE-PG} with $\sigma_k = k^{-\beta}$ for some $\beta\in (0, 1)$. Let $x^*\in\ubar{X}$ and $D$ be the bound on the sequence $\{\norm{x^k-x^*}
    \}_{k\in \N_0}$ defined in Lemma~\ref{lem:bounded}.
    Then, for any $K\in\N$, $\mathcal{K}=\{K+1,K+2,\ldots,2K\}$, and $$K^*:=\argmin_{k\in\mathcal{K}}\bigg\{t_k(\varphi(x^{k}) - \varphi^*)+t_k\sigma_{k}(\omega(x^k) - \omega^*)\bigg\}$$    
    the following inequalities hold:
    \begin{align*}\omega(x^{K^*}) - \omega^*&\leq \frac{D^2\alpha_1}{K^{1-\beta}},%
    \quad&\varphi(x^{K^*}) - \varphi^*\leq
    \frac{D^2\alpha_2}{2K}+\frac{\Delta_\omega}{K^{\beta}},
\end{align*} where $\alpha_1=L_1+L_2$, and $\alpha_2=L_1+L_2$ if $t_k$ is chosen as a constant step size, and  $\alpha_1=\max\{\bar{t}^{-1},\gamma^{-1}(L_1+L_2)\}$, and $\alpha_2=\max\{\gamma^{-1}(L_1+L_2),\bar{t}^{-1}\}$  if $t_k$ is chosen through Algorithm \ref{alg:backtracking}.
\end{theorem}
\begin{proof}
 By the definition of $K^*$, applying Lemma \ref{lem:IRE-PG} to $k\in \mathcal{K}$
\begin{eqnarray}
    2t_{K^*}(\varphi(x^{K^*}) - \varphi^*)+2\sigma_{K^*}t_{K^*}(\omega(x^{K^*}) - \omega^*)\nonumber
    &\leq& 
    2t_k(\varphi(x^{k}) - \varphi^*)+2t_k\sigma_k(\omega(x^{k}) - \omega^*)\nonumber\\
    &\leq&\|x^{k-1} - x^*\|^2 - \|x^{k} - x^*\|^2.\label{3.13}
    \end{eqnarray}
    Summing \eqref{3.13} over $k\in\mathcal{K}$ and dividing by $2K$, we get
    \begin{equation}
       t_{K^*}(\varphi(x^{K^*}) - \varphi^*)+\sigma_{K^*}{t_{K^*}}(\omega(x^{K^*}) - \omega^*)\leq \frac{\norm{x^{K}-x^*}^2}{2K}\leq \frac{D^2}{2K}.\label{3.14}
    \end{equation}
    We start by proving the convergence rate with respect to $\omega$.
    If $\omega(x^{K^*}) - \omega^*\leq 0$ then the result trivially holds. Suppose  $\omega(x^{K^*}) - \omega^*> 0$. Then, recalling that $\sigma_{K^*}t_{K^*}=\pi_{K^*}\geq \ubar{\pi}_{K^*}=\ubar{t}_{K^*}\sigma_{K^*}$  (see the proof of Theorem~\ref{them:conv_rate_IREPG} for the definition of $\ubar{t}_k$), with $\ubar{\pi}_k$ being nonincreasing, we have that $\pi_{K^*}\geq \ubar{\pi}_{2K}.$ Thus, using \eqref{3.14}, and $\varphi(x^{K^*})\geq \varphi^*$, 
    and a similar derivation to \eqref{eq:bound_on_sum_pi} we obtain that \begin{equation*}\omega(x^{K^*}) - \omega^*\leq \frac{D^2}{2}\frac{1}{K\sigma_{2K}\ubar{t}_{2K}}\leq \frac{D^2}{2}\frac{\alpha_1}{2^{-\beta}K^{1-\beta}}\leq \frac{D^2\alpha_1}{K^{1-\beta}}.\end{equation*}
  Moving to the convergence of $\varphi$, it follows from \eqref{3.14} that
     \begin{align*}
\label{eq:bound_varphi_Kstar1}\varphi(x^{K^*})-\varphi^*&\leq \frac{D^2}{2Kt_{K^*}}+(\omega^*-\omega(x^{K^*}))\sigma_{K^*}\leq \frac{D^2\alpha_2}{2K}+\frac{\Delta_\omega}{K^{\beta}},
     \end{align*}
where the second inequality is due to $\omega^*-\omega(x^k)\leq \Delta_\omega$, the fact that $\sigma_{K^*}\leq \sigma_{K}=K^{-\beta}$, the definition of $\alpha_2$ and the fact that
\begin{equation*}\frac{1}{t_{K^*}}\leq\begin{cases}
    L_1 + L_2, & t_k \text{ is a constant step size},\\
    \max\left\{\frac{L_1+L_2}{\gamma},\frac{1}{\bar{t}}\right\}, & t_k \text{ is a backtracking step size}.
\end{cases}\end{equation*}
\end{proof}
Thus, looking at $K^*$ enables us to extend the trade-off between the convergence rates of the {inner and outer} functions beyond both our ergodic result in Theorem~\ref{them:conv_rate_IREPG} and those obtained by Bi-SG in \cite{merchavshoham2023}.  Theorem~\ref{thm:conv_rate_IREPG2} also allows us to utilize this trade-off to address situations where computing the proximal operator over $\sigma_kg_1+g_2$ is computationally expensive.
Specifically, applying {Proposition~\ref{prop:interpert_rate}\ref{prop:interpert_rate_iii}} to the rate results obtained in Theorem~\ref{thm:conv_rate_IREPG2} we can conclude the following.

\begin{cor}\label{cor:ire_pg}
    Let $\{w^k\equiv(x^k,p^k)\}_{k\in\N}$ be the sequence generated by \ref{alg:IRE-PG} applied to problem~\eqref{eq:BLO_surrogate} with $\sigma_k=\frac{1}{k^{2/3}}$. 
    If $\omega$ is coercive and real-valued, then {for} any  $x^*\in\ubar{X}$ there exist constants $\theta_1>0$ and $\theta_2>0$ such that for any $K\in \N$ and $K^*$ as defined in Theorem~\ref{thm:conv_rate_IREPG2} the iterate $x^{K^*}$ satisfies
 \begin{align*}
     \omega({x}^{K^*})-\omega(x^*)\leq \frac{\theta_1}{K^{1/3}}%
     ,\quad \varphi({x}^{K^*})-\varphi^*\leq \frac{\theta_2}{K^{2/3}}.%
 \end{align*}
\end{cor}

We note that Theorem~\ref{thm:conv_rate_IREPG2} provides stronger convergence results than what we are able to obtain for the ergodic sequence. However, the significance of this result is mainly theoretical, since computing $K^*$ for increasing values of $K$ requires storing an increasing amount of iterations and is therefore not implementable in practice.

\section{Iteratively REgularized Accelerated Proximal Gradient}\label{sec:five}
   In this section, we present and analyze an accelerated version of \ref{alg:IRE-PG}, based on FISTA. The IRE-APG is given in Algorithm~\ref{alg:IRE-APG}, where we consider a constant step size, as well as a modified backtracking procedure.
   \begin{varalgorithm}{IRE-APG}
    \caption{
    Iteratively REgularized Accelerated Proximal Gradient}\label{alg:IRE-APG}
    \begin{enumerate}
        \item[  ] {\bf Parameters}: A sequence $\{\sigma_{k}\}_{k\in\N}$, scalar $\bar{t}>0$, and an initial point $x^0\in \R^n$.
        \item[ ] {\bf Initialization}: $y^0=x^0\in \R^n$, $s_0=1$, and $t_0=\bar{t}$.
        \item[ ] {\bf General Step}: For $k=1,2,\dots$
        \begin{enumerate}[label=\arabic*.]
            \item%
            Set $F_k=\sigma_kf_1 + f_2$, $G_k=\sigma_kg_1 + g_2$,  and choose $t_k$ from these options
            \begin{enumerate}[label={(\alph*)}]
            \item $t_k=\frac{1}{L_2+\sigma_kL_1}$ (constant step size).
                \item $t_k=\mathcal{B}^{F_k}_{G_k}(y^{k-1},t_{k-1})$ defined by Algorithm~\ref{alg:backtracking} (backtracking step size).
            \end{enumerate} 
            \item Compute \vskip-30pt
    \begin{align*}%
            x^{k} &= prox_{t_{k}G_k}(y^{k-1} - t_{k}\nabla F_k(y^{k-1})),\\
            y^{k}&=x^{k}+\frac{s_{k-1}-1}{s_k}(x^{k} - x^{k-1}),%
            \end{align*}
        \vskip-10pt where the sequence $\{s_k\}_{k\in\N}$ satisfies $s_{k} = \frac{1+\sqrt{1+4s^2_{k-1}}}{2}$ %
        \end{enumerate}
    \end{enumerate}    
\end{varalgorithm}
       
       In order to analyze the convergence rate of Algorithm~\ref{alg:IRE-APG} for some optimal solution $x^*\in\ubar{X}$ to problem~\eqref{Bilevel}, we define $v^k := \varphi(x^k) - \varphi^*$, $z^k := \omega(x^k) - \omega^*$, and $u^k:= (s_{k-1}x^k-(s_{k-1}-1)x^{k-1}-x^*)$,
       which we use to prove the following auxiliary results needed for the analysis. %
       
\begin{lemma}\label{lem:rec_apg}
            Let $\{x^k\}_{k\in\N}$ be the sequence generated by \ref{alg:IRE-APG}, and let  $x^*\in \ubar{X}$. 
\begin{enumerate}[label=(\roman*)]
\item             
If $t_k$ is chosen as a constant step size, then for any $K\in \mathbb{N}$, 
\begin{equation*}\hspace{-22pt}\frac{1}{2t_{K}}\|u^{K}\|^2 + s^2_{K-1}v^{K}\leq \frac{(L_1 + L_2)}{2}\norm{x^0-x^*}^2+\sum_{k=1}^{K-1}s^2_{k-1}z^k(\sigma_{k+1} - \sigma_{k}) - \sigma_Ks^2_{K-1}z^{K}.\end{equation*}
\item If $t_k$ is chosen through Algorithm \ref{alg:backtracking}, then for every $K\in \mathbb{N}$,
\begin{equation*}\hspace{-22pt}\frac{1}{2}\|u^{K}\|^2 + {s^2_{K-1}}t_Kv^{K}\leq\frac{1}{2}\norm{x^0-x^*}^2+\sum_{k=1}^{K-1}s_{k-1}^2z^k(\sigma_{k+1}t_{k+1}-\sigma_kt_k)-\sigma_{K}t_{K}s_{K-1}^2z^K.\end{equation*}
\end{enumerate}
       \end{lemma}
       \begin{proof}
             {We start with the case $K=1$. Since $x^0=y^0$, applying Lemma~\ref{lem:property_fista_seq} with  
        $x = x^*$ results in the inequality
       \begin{equation}
       v^{1} + \sigma_{1}z^1=\phi_{1}(x^1)-\phi_{1}(x^*)\leq \frac{1}{2t_1}(\norm{x^0-x^*}^2-\norm{x^1-x^*}^2).\label{proxgrad}
       \end{equation}
       By definition, $u^1=s_0x^1-(s_0-1)x^0-x^*=x^1-x^*$. Thus, from \eqref{proxgrad} we get that 
       \begin{eqnarray}
       \frac{1}{2}\|u^{1}\|^2+t_1v^1+t_1\sigma_{1}z^1 &\leq& \frac{1}{2}\|x^0-x^*\|^2.\label{first_step_fista}
       \end{eqnarray}}
         
           Our next two inequalities follow the general convergence analysis of FISTA. Using Lemma~\ref{lem:property_fista_seq} with $x=
           x^*s_{k-1}^{-1}+x^{k-1}(1-s_{k-1}^{-1})$ and the definition of $y^{k-1}$
           we obtain
            \begin{align}
            &t_k(\phi_k(x^*s_{k-1}^{-1}+x^{k-1}(1-s_{k-1}^{-1}))-\phi_k(x^k))\nonumber\\
            &\geq 
            \frac{1}{2}\left(\norm{x^k-x^*s_{k-1}^{-1}-x^{k-1}(1-s_{k-1}^{-1})}^2-\norm{y^{k-1}-x^*s_{k-1}^{-1}-x^{k-1}(1-s_{k-1}^{-1})}^2\right)\nonumber\\
            &=\frac{1}{2s_{k-1}^2}\left(\norm{s_{k-1}x^k-x^*-x^{k-1}(s_{k-1}-1)}^2-\norm{s_{k-2}x^{k-1}-x^*-(s_{k-2}-1)x^{k-2}}^2\right)\nonumber\\
            &=\frac{1}{2s_{k-1}^2}(\norm{u^{k}}^2-\norm{u^{k-1}}^2)
            \label{eq:lb_recrs_apg}
            \end{align}
            From convexity of $\phi_k$ we have that
            \begin{align}&\phi_k(x^*s_{k-1}^{-1}+x^{k-1}(1-s_{k-1}^{-1}))-\phi_k(x^k)\nonumber\\
            &\leq s_{k-1}^{-1}\phi_k(x^*)+(1-s_{k-1}^{-1})\phi_k(x^{k-1})-\phi_k(x^k)\nonumber\\
            &=\frac{1}{s_{k-1}^2}\big(s_{k-1}(\phi_k(x^*)-\phi_k(x^{k-1}))+s_{k-1}^2(\phi_k(x^{k-1})-\phi_k(x^{k}))\big)\nonumber\\
           &= \frac{1}{s_{k-1}^2}(s_{k-2}^2v^{k-1}-s_{k-1}^2v^k+\sigma_k(s_{k-2}^2z^{k-1}-s_{k-1}^2z^{k})),\label{eq:ub_recrs_apg}
        \end{align}
         where the last equality follows from $s_{k-1}^2-s_{k-1}=s_{k-2}^2$, and the definitions of  $\phi_k$, $v^k$, and $z^k$. Combining \eqref{eq:lb_recrs_apg} and \eqref{eq:ub_recrs_apg} we get that
for all $k\in \N$, $k\geq 2$,
   \begin{eqnarray}
 \frac{1}{2t_{k}}\|u^{k}\|^2 + s^2_{k-1}v^{k}\leq 
\frac{1}{2t_{k}}\|u^{k-1}\|^2 + s^2_{k-2}v^{k-1} + \sigma_k(s^2_{k-2}z^{k-1} - s^2_{k-1}z^{k}).\label{2.28}
   \end{eqnarray}

   In the \emph{constant step size} case, $t_{k}\geq t_{k-1}$, and thus \eqref{2.28} transforms into 
   \begin{eqnarray}
 \frac{1}{2t_{k}}\|u^{k}\|^2 + s^2_{k-1}v^{k}\leq 
\frac{1}{2t_{k-1}}\|u^{k-1}\|^2 + s^2_{k-2}v^{k-1} + \sigma_k(s^2_{k-2}z^{k-1} - s^2_{k-1}z^{k}).\label{eq:const_step_recurse_apg}
   \end{eqnarray}
   Summing \eqref{eq:const_step_recurse_apg} over  $k=2,\dots,K$ and using \eqref{first_step_fista}, we get
   \begin{eqnarray}
      \frac{1}{2t_{K}}\|u^{K}\|^2 + s^2_{K-1}v^{K}
      &\leq& \frac{1}{2t_{1}}\|u^{1}\|^2 + s^2_{0}v^{1}+\sum_{k=2}^K\sigma_k(s^2_{k-2}z^{k-1}-s^2_{k-1}z^{k})\nonumber\\
      &\leq&  \frac{L_1+L_2}{2}\|x^{0}-x^*\|^2 +\sum_{k=1}^{K-1}s^2_{k-1}z^k(\sigma_{k+1} - \sigma_{k}) - \sigma_{K}s^2_{K-1}z^{K}.\label{2.29}
   \end{eqnarray}
In the \emph{backtracking step size} case, $t_{k-1}\geq t_{k}$, and \eqref{2.28}
is equivalent to \begin{equation}
\frac{1}{2}\|u^{k}\|^2 + t_ks^2_{k-1}v^{k}\leq 
\frac{1}{2}\|u^{k-1}\|^2 + t_{k-1}s^2_{k-2}v^{k-1} + t_k\sigma_k(s^2_{k-2}z^{k-1} - s^2_{k-1}z^{k}).\label{eq:backtracking_recurse_apg}
\end{equation}
Summing \eqref{eq:backtracking_recurse_apg} over $k=\{2,3,\ldots,K\}$ and using \eqref{first_step_fista}, we get
\begin{align}
     \frac{1}{2}\|u^{K}\|^2 + {s^2_{K-1}}{t_{K}}v^{K}%
     &\leq \frac{1}{2}\|u^{1}\|^2 + s^2_{0}v^{1}t_{1}+\sum_{k=2}^K\sigma_kt_k(s^2_{k-2}z^{k-1}-s^2_{k-1}z^{k})\nonumber\\
     &\leq\frac{1}{2}\norm{x^0-x^*}^2\hspace{-2pt}+\hspace{-2pt}\sum_{k=1}^{{K-1}}s_{k-1}^2z^k(\sigma_{k+1}t_{k+1}-\sigma_kt_k)-\sigma_{K}t_{K}s_{K-1}^2z^K.\nonumber\\
     \nonumber
\end{align}
\end{proof}

Our {next} result involves the convergence of $\varphi(x^k)$ to $\varphi^*$ (feasibility gap) for a specific choice of $\sigma_k$ similar to \ref{alg:IRE-PG}. %
\begin{prop}\label{prop:vbound}
  Let {$x^*\in \ubar{X}$, and} let $\{x^k\}_{k\in\N}$ be the sequence generated by \ref{alg:IRE-APG} with $\sigma_k = k^{-\beta}$, for some  $\beta\in (0,2]$. %
  Then for all $K\in \mathbb{N}$,
  \begin{equation*}
      \varphi(x^{K}) - \varphi^*\leq \frac{2\alpha_1\norm{x^0-x^*}^2}{(K+1)^2}+ \frac{4\alpha_2(\alpha_3+\alpha_4\log(K))\Delta_\omega}{(K+1)^\beta},
  \end{equation*}
  where 
  $\alpha_1=L_1+L_2$ and $\alpha_2=2$ if $t_k$ is chosen as a constant step size,  or alternatively ${\alpha_1=
\max\{\gamma^{-1}(L_1+L_2),\bar{t}^{-1}\}}$ and $\alpha_2=\alpha_1\bar{t}$
if $t_k$ is chosen through Algorithm \ref{alg:backtracking}, and
  \begin{equation*}
            \alpha_3=\begin{cases}
                {\frac{1}{2-\beta}}, &\beta\in(0,2),\\
                1, & \beta=2,
            \end{cases}\text{ and }
            \alpha_4=\begin{cases} 0, & \beta\in(0,2),\\
           1, & \beta=2.\end{cases}
        \end{equation*}
\end{prop}
\begin{proof}
    We start by analyzing the case where $t_k$ is set as a constant step size. Using  Lemma~\ref{lem:rec_apg}(i) to bound $v^k=\varphi(x^k)-\varphi^*$, we start bounding the term on the right-hand side of the inequality. Recall that $-z^k = \omega^*-\omega(x^k)  \leq \Delta_\omega$ for each $k$, 
   where $\Delta_\omega\geq 0$. Therefore, since $s_{k-1}\leq k$ by Lemma \ref{lem:seq},
   it follows that 
   \begin{eqnarray}
       \sum_{k=1}^{K-1}s^2_{k-1}z^k(\sigma_{k+1} - \sigma_{k})-\sigma_Ks^2_{K-1}z^{K}&\leq& (\sum_{k=1}^{K-1}s^2_{k-1}(\sigma_{k} - \sigma_{k+1})+\sigma_Ks^2_{K-1})\Delta_\omega\nonumber\\
       &\leq&(\sum_{k=1}^{K-1}k^2(\frac{1}{k^\beta} - \frac{1}{(k+1)^\beta})+\frac{K^2}{K^\beta})\Delta_\omega\nonumber\\
       &=&(\sum_{k=1}^{K} \frac{1}{k^\beta}(k^2-(k-1)^2))\Delta_\omega\leq \Delta_\omega\sum_{k=1}^{K}2k^{1-\beta}.\label{2.30b}
   \end{eqnarray} 
   When $t_k$ is chosen as constant step size, Lemma \ref{lem:rec_apg}(i) together with \eqref{2.30b} yields
   \begin{eqnarray}
       v^{K}=\varphi(x^{K}) - \varphi(x^*)&\leq& \frac{2(L_1 + L_2)\norm{x^0-x^*}^2}{(K+1)^2}+ \frac{ 8\Delta_\omega\sum_{k=1}^{K}k^{1-\beta}}{(K+1)^2},\label{2.31b}
   \end{eqnarray}
   where we used $s_{K-1}\geq {(K+1)}/{2}$ from Lemma~\ref{lem:seq} to obtain the inequality. Employing Lemma~\ref{lem:sum_bounds}(i) to upper bound $\sum_{k=1}^{K}k^{1-\beta}$ we obtain the desired result. 
   
   Alternatively, if $t_k$ is determined by Algorithm \ref{alg:backtracking}, we can use Lemma~\ref{lem:rec_apg}(ii) to bound $v_k$. We again start by bounding the right-hand-side term by \begin{eqnarray}
   &&\sum_{k=1}^{K-1}s_{k-1}^2z^k(\sigma_{k+1}t_{k+1}-\sigma_kt_k)-\sigma_{K}t_Ks_{K-1}^2z^K\nonumber\\
   &&\leq 
   (\sum_{k=1}^{K-1}s_{k-1}^2(\sigma_kt_k-\sigma_{k+1}t_{k+1})+\sigma_{K}t_Ks_{K-1}^2)\Delta_\omega\nonumber\\
   &&=(\sigma_1t_1s^2_0 +\sum_{k=2}^K\sigma_kt_{k}(s^2_{k-1} - s^2_{k-2}))\Delta_\omega\nonumber\\
&&=\sum_{k=1}^K\sigma_kt_{k}s_{k-1}\Delta_\omega\leq t_1\Delta_\omega\sum_{k=1}^Kk^{1-\beta},\label{2.30c}
   \end{eqnarray}
   where the first inequality follows from $t_k\geq t_{k+1}$, $\sigma_k\geq \sigma_{k+1}$, and $-z^k\leq \Delta_\omega$, the second equality follows from $s^2_{k-1}-s^2_{k-2} = s_{k-1}$ and $s_0=1$,  and the second inequality is due to $s_{k-1}\leq k$ from Lemma \ref{lem:seq}. Plugging \eqref{2.30c} into Lemma~\ref{lem:rec_apg}(ii) %
   we obtain
   \begin{eqnarray*}
       v^K&\leq& \frac{\norm{x^0-x^*}^2}{2s^2_{K-1}t_K} + \frac{t_1\Delta_\omega}{s^2_{K-1}t_K}\sum_{k=1}^Kk^{1-\beta}%
       \leq \frac{2\alpha_1\norm{x^0-x^*}^2}{(K+1)^2} + \frac{4\alpha_1\bar{t}\Delta_\omega}{(K+1)^2}\sum_{k=1}^K k^{1-\beta},
   \end{eqnarray*}
   where the last inequality follows from applying $s_{K-1}\geq (K+1)/2$ from Lemma~\ref{lem:seq}, the fact that $t_K\geq \min\{\gamma(L_1\sigma_K+L_2)^{-1},t_1\}\geq \alpha_1^{-1}$, as well as $t_1\leq \bar{t}$.
   Applying Lemma~\ref{lem:sum_bounds}(i) to  the last inequality gives the desired result.
\end{proof}
We are now ready to prove the convergence of the ergodic sequences generated by \ref{alg:IRE-APG}. We note that since this analysis relies on Lemma~\ref{lem:rec_apg}, it varies depending on the type of step size rule chosen.  Theorem~\ref{thm:rate_apg_constant_step} summarizes the convergence rate result in both cases, but the proof for the backtracking case is relegated to the appendix.
\begin{theorem}%
\label{thm:rate_apg_constant_step}  Let $\{x^k\}_{k\in\N}$ be the sequence generated by \ref{alg:IRE-APG} with $\sigma_k = k^{-\beta}$, $\beta\in (0, 2]$. Moreover, let
$\bar{x}^K = \sum_{k=1}^{K}\frac{\pi_k}{\sum_{l=1}^{K}\pi_l}x^k$, where $\pi_k$ is defined as follows:
\begin{enumerate}[label=(\alph*)]
\item If $t_k$ is defined as a constant step size,          \begin{equation*}
  \pi_k:=\begin{cases}
                s^2_{k-1}(\sigma_{k}-\sigma_{k+1}),& 1\leq k\leq {K-1},\\
                \sigma_Ks^2_{K-1}, & k = K,
            \end{cases}
        \end{equation*} $\tilde{\alpha}_1 =(L_1+L_2)$, $\tilde{\alpha}_2= 2(L_1+L_2)\|x^0 - x^*\|^2 +{\frac{8}{2-\beta}}\Delta_\omega$, $\tilde{\alpha}_3 = 2(L_1+L_2)\|x^0 - x^*\|^2 + 8\Delta_\omega$ and $\tilde{\alpha}_4=8\Delta_\omega$.  
\item If $t_k$ is chosen through Algorithm \ref{alg:backtracking},
\begin{equation*}
        {\pi}_k:=\begin{cases}
                s^2_{k-1}(\sigma_{k}t_k-\sigma_{k+1}t_{k+1}),& 1\leq k\leq {K-1},\\
            \sigma_Kt_Ks^2_{K-1}, & k = K,
            \end{cases}
        \end{equation*}
$\tilde{\alpha}_1=\max\{(L_1+L_2)\gamma^{-1}, \bar{t}^{-1}\}$,
$\tilde{\alpha}_2 = 2\tilde{\alpha}_1^2\bar{t}(\|x^0 - x^*\|^2 + \frac{2\bar{t}}{(2-\beta)}\Delta_\omega)$,$\tilde{\alpha}_3 = 2\tilde{\alpha}_1^2\bar{t}(\|x^0 - x^*\|^2 + 2\bar{t}\Delta_\omega)$ and $\tilde{\alpha}_4 = 4(\tilde{\alpha}_1\bar{t})^2\Delta_\omega$.
\end{enumerate}
Then, the following inequalities hold for any  $x^*\in\ubar{X}$ and $K\in\N$: 
        \begin{enumerate}[label=(\roman*)]
            \item \; \vspace{-1.5em} \begin{equation*}
         \omega(\bar{x}^K) - \omega^*\leq\begin{cases}
                \frac{2\tilde{\alpha}_1}{K^{2-\beta}}\|x^0 - x^*\|^2,& \beta\in (0,2),\\
                \frac{\tilde{\alpha}_1}{\log(K+1)}\|x^0 - x^*\|^2,& \beta = 2.%
            \end{cases}
        \end{equation*}
            \item \; \vspace{-1.5em} \begin{equation*}
         \varphi(\bar{x}^K) - \varphi(x^*)\leq\begin{cases}
         \frac{16\tilde{\alpha}_2}{(1-\beta)K^{\beta}},& \beta\in (0,1),\\      \frac{8\tilde{\alpha}_2(1+\log(K))}{K}, &\beta = 1,\\
         \frac{4(2\beta-1)\tilde{\alpha}_2}{(\beta-1)K^{2-\beta}}, &\beta \in (1,2),\\\frac{2(4\tilde{\alpha}_3+\tilde{\alpha}_4)}{\log(K+1)}, & \beta = 2.
            \end{cases}
        \end{equation*}
        \end{enumerate}
\end{theorem}
\begin{proof}
As mentioned above, we give the proof for the case where $t_k$ is chosen to be a constant step size; the proof for the case of backtracking is given in the appendix.
\begin{enumerate}[label=(\roman*)]
     \item By the convexity of $\omega$ and from Lemma \ref{lem:rec_apg}(i) we have that
    \begin{eqnarray}
(\sum_{k=1}^{K}\pi_k)(\omega(\bar{x}^K) - \omega^*)&\leq&\sum_{k=1}^{K}\pi_k(\omega(x^k) - \omega^*) =\sum_{k=1}^{K}\pi_kz^k\nonumber\\
&\leq &\frac{L_1+L_2}{2}\norm{x^0-x^*}^2-\frac{1}{2t_{K}}\|u^{K}\|^2 - s^2_{K-1}v^{K}\nonumber\\
    &\leq&\frac{L_1+L_2}{2}\norm{x^0-x^*}^2.\label{2.34}
    \end{eqnarray}
Observe that
    \begin{eqnarray}
      \sum_{k=1}^{K}\pi_k &=& \sum_{k=1}^{K-1}s^2_{k-1}(\sigma_{k} - \sigma_{k+1}) + \sigma_Ks^2_{K-1}\nonumber\\
      &\geq& \sum_{k=1}^{K-1}(\frac{k+1}{2})^2(\frac{1}{k^\beta}-\frac{1}{(k+1)^\beta})+\frac{(K+1)^2}{4K^\beta}\nonumber\\
      &=& 1+\sum_{k=2}^{K} \frac{1}{4k^\beta}((k+1)^2-k^2)\nonumber\\  &\geq&\sum_{k=1}^{K}\frac{2k+1}{4k^\beta}\geq \sum_{k=1}^{K}\frac{k^{1-\beta}}{2}%
      \geq\begin{cases}
       \frac{K^{2-\beta}}{4}, &\beta\in(0,2), \\
        \frac{1}{2}\log (K+1), & \beta=2,%
   \end{cases}\label{2.33}
    \end{eqnarray}
        where \eqref{2.33} 
 follows from Lemma \ref{lem:sum_bounds}(ii). Substituting \eqref{2.33} in \eqref{2.34} gives the desired result.
\item Proposition \ref{prop:vbound} implies that for all $k\in \N$, $\varphi(x^k) - \varphi^*\leq \frac{\tilde{\alpha}_2}{(k+1)^\beta}$ if $\beta\in(0,2)$, and $\varphi(x^k) - \varphi(x^*)\leq \frac{\tilde{\alpha}_3+\tilde{\alpha}_4\log(k)}{(k+1)^2}$ if $\beta = 2$. Thus, by the convexity of $\varphi$, the definition of $\pi_k$, and $s_{k-1}\leq k$ from   Lemma~\ref{lem:seq}, we obtain that for $\beta\in(0,2)$
    \begin{eqnarray}
    \varphi(\bar{x}^K) - \varphi(x^*)&\leq&\frac{1}{\sum_{k=1}^{K}\pi_k} \sum_{k=1}^{K}\pi_k(\varphi(x^k) - \varphi(x^*)) \nonumber\\
    &\leq& \frac{4\tilde{\alpha}_2}{K^{2-\beta}}\left(\sum_{k=1}^{K-1}\frac{k^2}{(k+1)^\beta}\bigg(\frac{1}{k^\beta} - \frac{1}{{(k+1)}^\beta}\bigg) + \frac{K^2}{K^{\beta}(K+1)^\beta}\right)\nonumber\\
    &= & \frac{4\tilde{\alpha}_2}{K^{2-\beta}}\sum_{k=1}^{K}\frac{1}{k^\beta}\big(\frac{k^{2}}{(k+1)^\beta} - \frac{(k-1)^{2}}{k^\beta}\big)\nonumber\\
    &\leq& \frac{4\tilde{\alpha}_2}{K^{2-\beta}} \sum_{k=1}^{K}\left(\frac{k^{2}}{k^{2\beta}} - \frac{(k-1)^{2}}{k^{2\beta}}\right)%
    \leq \frac{8\tilde{\alpha}_2}{K^{2-\beta}}\sum_{k=1}^{K}k^{1-2\beta}.\label{2.34a}
    \end{eqnarray} 
    Applying Lemma \ref{lem:sum_bounds}(i) {and $(K+1)^{2-2\beta}\leq 4K^{2-2\beta}$} to \eqref{2.34a} we obtain (ii) for $\beta\in(0,2)$.
    Similarly, when $\beta=2$, applying the definition of $\pi_k$ and $s_{k-1}\leq k$ from Lemma~\ref{lem:seq}, we obtain
    \begin{eqnarray}
    (\sum_{k=1}^{K}\pi_k)(\varphi(\bar{x}^K) - \varphi(x^*))&\leq& \sum_{k=1}^{K}\pi_k(\varphi(x^k) - \varphi(x^*))\nonumber\\
    &\leq& \tilde{\alpha}_3\left(\sum_{k=1}^{K-1} \frac{k^2}{(k+1)^2}\bigg(\frac{1}{k^2} - \frac{1}{{(k+1)}^2}\bigg) + \frac{1}{(K+1)^2}\right)\nonumber\\
    & + & \tilde{\alpha}_4\left(\sum_{k=1}^{K-1}\frac{k^2\log k}{(k+1)^2}\bigg(\frac{1}{k^2} - \frac{1}{{(k+1)}^2}\bigg) + \frac{\log (K)}{(K+1)^2}\right).\label{vgenr}
    \end{eqnarray}
    Bounding the terms multiplying $\tilde{\alpha}_4$, we have
    \begin{eqnarray}
        \sum_{k=1}^{K-1}\frac{k^2\log k}{(k+1)^2}\bigg(\frac{1}{k^2} - \frac{1}{{(k+1)}^2}\bigg) + \frac{\log (K)}{(K+1)^2}
        &\leq& \sum_{k=1}^{K-1}\frac{(2k+1)\log k}{(k+1)^4}+\frac{\log(K)}{(K+1)^2}\nonumber\\
        &\leq& \sum_{k=1}^{K-1}\frac{2\log(k)}{(k+1)^3}+\frac{\log(K)}{(K+1)^2}\leq 1,\nonumber%
        \end{eqnarray}
    where we have used Lemma \ref{lem:sum_bounds}(iii) to obtain the last inequality.
   As for the multipliers of $\tilde{\alpha}_3$, using Lemma~\ref{lem:sum_bounds}(i) for $\beta=4$ results in
     \begin{eqnarray}
     \sum_{k=1}^{K-1}\hspace{-2pt}\frac{k^2}{(k+1)^2}\bigg(\frac{1}{k^2}\hspace{-2pt} - \hspace{-2pt}\frac{1}{{(k+1)}^2}\bigg)\hspace{-2pt} %
     =\hspace{-2pt}\sum_{k=1}^{K-1}\frac{(2k+1)}{(k+1)^4}%
     \leq 2\sum_{k=1}^{K+1} k^{-3} \leq 3 \nonumber,
     \end{eqnarray}
   Plugging these bounds together with $1/(K+1)^2\leq 1$ %
    into \eqref{vgenr}, we get
      \begin{equation}
    (\sum_{k=1}^{K}\pi_k)(\varphi(\bar{x}^K) - \varphi(x^*))
     \leq  4\tilde{\alpha}_3 + \tilde{\alpha}_4,\label{ratev6} 
     \end{equation}
    which together with \eqref{2.33} produces the desired result for this case.
\end{enumerate}      
\end{proof}
\remark{As in Theorem~\ref{them:conv_rate_IREPG}, Theorem~\ref{thm:rate_apg_constant_step} does not require the boundedness of $\ubar{X}$ in Assumption~\ref{ass:omega_min}. However, as mentioned in Remark~\ref{rem:thm_IRE_conv}, Assumption~\ref{ass:omega_min} guarantees that any limit point of the sequence $\{\bar{x}^k\}_{k\in\N}$ is in $\ubar{X}$.}

Theorem~\ref{thm:rate_apg_constant_step} 
provides a convergence rate of $\mathcal{O}(k^{-\beta})$ with respect to the inner function and $\mathcal{O}(k^{-(2-\beta)})$ with respect to the outer function for $\beta\in(0,1)$ with an additional $\log(k)$ term for the inner function when $\beta=1$. Note that analogously to the result of Theorem~\ref{them:conv_rate_IREPG}, in this range the convergence rate of the outer function is always faster than  that of the inner function. When $\beta=1$, these rates coincide with the simultaneous convergence rates obtained in \cite{merchav2024}, which were derived for a variation of \ref{alg:IRE-APG} in which $\sigma_k=1/(k+a)$ and $s_k=(k+a)/a$ for some $a\geq 2$. The additional benefit of these results is that they elucidate the trade-offs between the two convergence rates that can be simultaneously obtained for the inner and outer functions for different $\beta$ values.

\begin{cor}\label{cor:apg}
    Let $\{w^k\equiv(x^k,p^k)\}_{k\in\N}$ be the sequence generated by \ref{alg:IRE-APG} applied to problem~\eqref{eq:BLO_surrogate}. For any $k\in\N$ let
    $\{\bar{w}^k=(\bar{x}^k,\bar{p}^k)\}_{k\in\N}$ be the ergodic sequence as defined in Theorem~\ref{thm:rate_apg_constant_step}. Suppose $\sigma_k={k}^{-1}$ and $\omega$ is coercive.
    If $\omega$ is additionally real-valued, then there exist constants $\theta_1>0$ and $\theta_2>0$ such that for any $k\in \N$ %
 \begin{align*}
     \omega(\bar{x}^k)-\omega^*\leq \frac{\theta_1(1+\log(k))^{1/2}}{k^{1/2}}%
     ,\quad \varphi(\bar{x}^{k})-\varphi^*\leq \frac{\theta_2(1+\log(k))}{k}.
 \end{align*}
 Alternatively, if $\varphi$ is real-valued, then there exist constants $\theta_3>0$ and $\theta_4>0$ such that for any $k\in \N$ %
 \begin{align*}
     \omega(\bar{p}^k)-\omega^*\leq \frac{\theta_3}{k},\quad \varphi(\bar{p}^{k})-\varphi^*\leq \frac{\theta_4(\log(k)+1)^{1/2}}{k^{1/2}}.
 \end{align*}
\end{cor}

Theorem~\ref{thm:conv_rate_IREPG2} shows that IRE-PG achieves a trade-off where the sum of the inner and outer exponents of $k$ equals $-1$, which is preserved with surrogate functions (Corollary~\ref{cor:ire_pg}). In contrast, IRE-APG accelerates this trade-off to $-2$ (Theorem~\ref{thm:rate_apg_constant_step}), but with surrogates (Corollary~\ref{cor:apg}) the sum improves only to $-1.5$. This limitation occurs because the accelerated trade-off in Theorem~\ref{thm:rate_apg_constant_step} is valid only for $\beta\in(0,1]$; extending this trade-off to $\beta\in(0,2]$ would allow full acceleration, but this is prevented by the difference between the analyses of the two algorithms.
We note that Corollary~\ref{cor:apg} does not exclude obtaining better rates for this case. Specifically, under additional assumptions on the functions, such as $\varphi$ admitting a H\"{o}lderian error bound, as shown in \cite{merchav2024}, or norm-like $\omega$, as shown in \cite{doronshtern2023} superior rates may be attained.

\section{Numerical Experiments}\label{sec:numetical_experiments}
In order to validate and demonstrate our theoretical results in the previous sections, we consider a simple instance of the motivating example discussed in Section~\ref{sec:problem_plus_example}.
Specifically, for some even integer $n$ we choose a signal $x\in\R^n$ such that its first $n/2$ components are equal to $-0.5$ and its last $n/2$ components are equal to $0.5$.
Given a matrix $A\in \R^{n/2,n}$, whose components are uniformly generated in $\{-1,0,1\}$, we produce a vector $y=Ax+\eta$ where $\eta\in\R^{n/2}$ is a noise vector that is sampled uniformly over a ball of radius $\tau$.
Assume that we know that 
$x\in C=[-1,1]^n$ and $\norm{Ax-y}\leq \tau$
and so our inner function is of the form of \eqref{eq:example_phi}.
Moreover, for the outer function we use
$\omega=\norm{Sx}_1$ where $S\in \R^{(n-1)\times n}$ is the matrix  associated with the total variation \cite{rudin_nonlinear_1992}, to promote signals with few jumps. While $\omega$ is nonsmooth and not proximal friendly, its structure allows us to use surrogate functions, as described in Remark~\ref{rem:lifting}.
Thus, in the following results we refer to the surrogate functions on the lifted variable $w=(x,p)\in \R^{2n-1}$ given by 
$\tilde{\omega}(w)=\norm{p}_1$ and
$\tilde{\varphi}(w)=\varphi(x)+\frac{\rho}{2}\norm{Sx-p}^2$. For comparison purposes, we also computed the optimal solution $x^*$ with values $\varphi^*=0$ and $\omega^*=0.997$ for this problem using CVXPy with its default solver \cite{diamond2016cvxpy}.

\begin{figure}[t]\caption{Performance of  sequences generated by IRE-PG and IRE-APG for different values of $\beta$} \label{fig:All_Results}
\includegraphics[width=\textwidth]{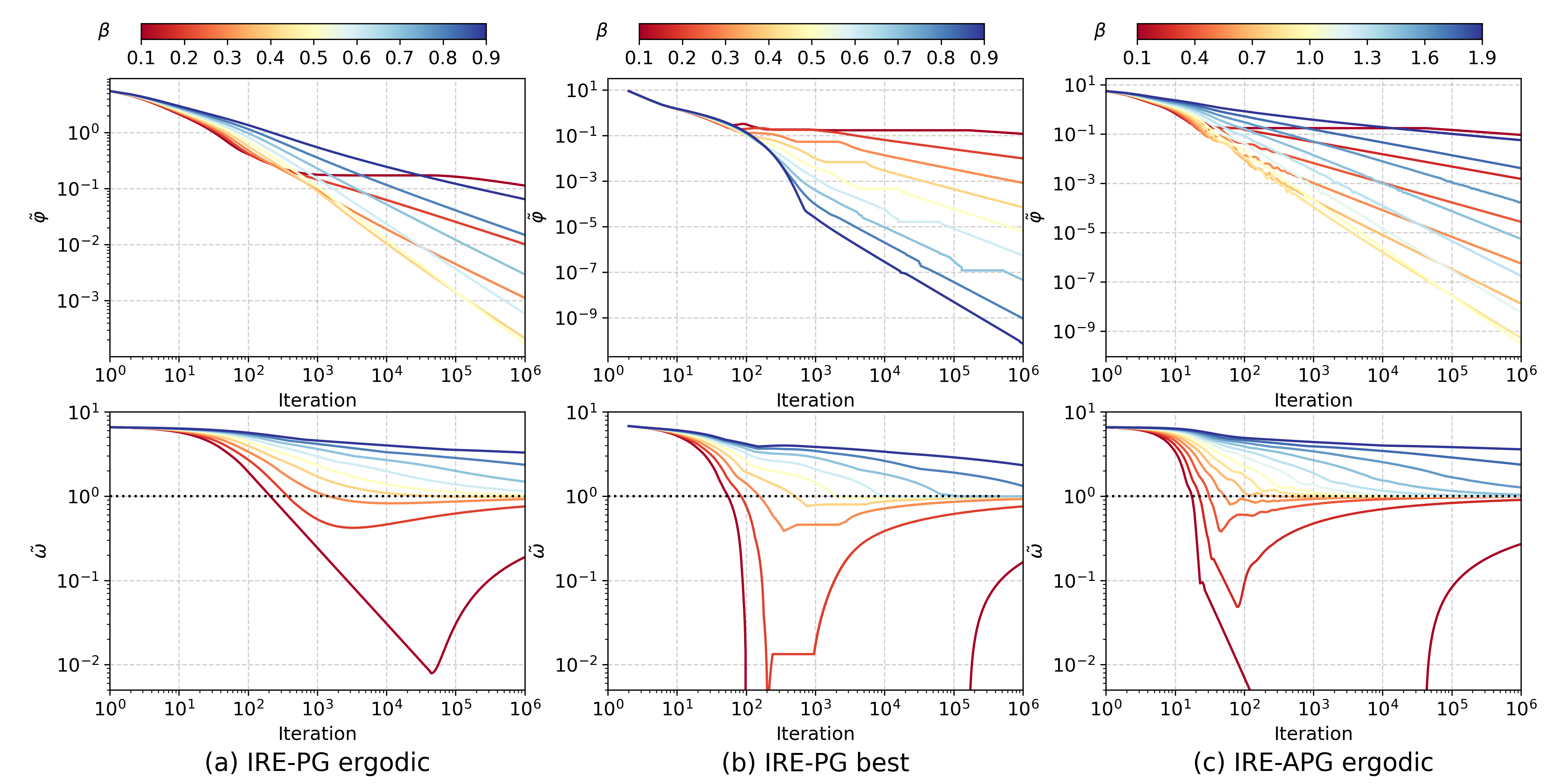}\vspace{-10pt}
\end{figure}

We set $\rho=1$ and ran both IRE-PG for $\beta\in \{0.1,0.2,\ldots,0.9\}$ and IRE-APG for $\beta\in \{0.1,0.25,\ldots,1.9\}$; the results %
are presented in Figure~\ref{fig:All_Results}. The first column depicts the ergodic performance of IRE-PG (i.e., the performance of the sequence $\{\bar{x}^k\}_{k\in\N}$). As expected, the convergence rate of $\tilde{\omega}$ to $\omega^*$ (depicted by the dotted line) improves as $\beta$ decreases, %
while the rate of convergence for $\tilde{\varphi}$ increases up to $\beta=0.5$ and then decreases again, validating %
Theorem~\ref{them:conv_rate_IREPG}. 
The ergodic performance of IRE-APG, displayed in the third column, has similar behavior,
validating Theorem~\ref{thm:rate_apg_constant_step}. Moreover, comparing the two methods, the superior trade-off of IRE-APG is apparent. %
The second column illustrates the performance of the sequence $x^{K^*}$ for IRE-PG. In contrast to the ergodic performance, we observe here a complete trade-off between the convergence rates of $\tilde{\varphi}$ and $\tilde{\omega}$ for $\beta \in (0,1)$. Specifically, the convergence rate of $\tilde{\varphi}$ improves as $\beta$ increases, confirming Theorem~\ref{thm:conv_rate_IREPG2}. Moreover, the ``best'' sequence consistently outperforms the ergodic sequence for both functions across all $\beta$.

Next, we examine how the surrogate function convergence rates translate to those of the original functions $\varphi$ and $\omega$ for various $\rho$ values. We consider three sequences: the ergodic IRE-PG sequence with $\beta=0.5$, the ``best'' IRE-PG sequence with $\beta=2/3$, and the ergodic IRE-APG sequence with $\beta=1$, with $\beta$ chosen in each case to optimize the trade-off between the inner and outer functions. Figure~\ref{fig:diff_rho} reports the values of $\norm{Sx-p}$ and the optimality gaps $\varphi-\varphi^*$ and $\omega-\omega^*$ for the three sequences. As expected from Remark~\ref{rem:rho_influence},  $\norm{Sx-p}$ decreases faster for all methods as $\rho$ increases, while the optimality gap of $\varphi$ grows due to its larger Lipschitz constant. %
Increasing $\rho$ also slows the convergence of $\tilde{\omega}$, but when mapped to $\omega(Sx)$ it benefits from the smaller difference between $Sx$ and $p$.
Consequently, although $\rho=10$ yields the slowest convergence, the difference between $\rho=0.1$ and $\rho=1$ is less significant, with $\rho=1$ eventually achieving the lowest optimality gap for $\omega$ across all methods.

\begin{figure}[t]\caption{Performance of IRE-PG and IRE-APG for different values of $\rho$}\label{fig:diff_rho}
\begin{center}
\includegraphics[width=\textwidth]{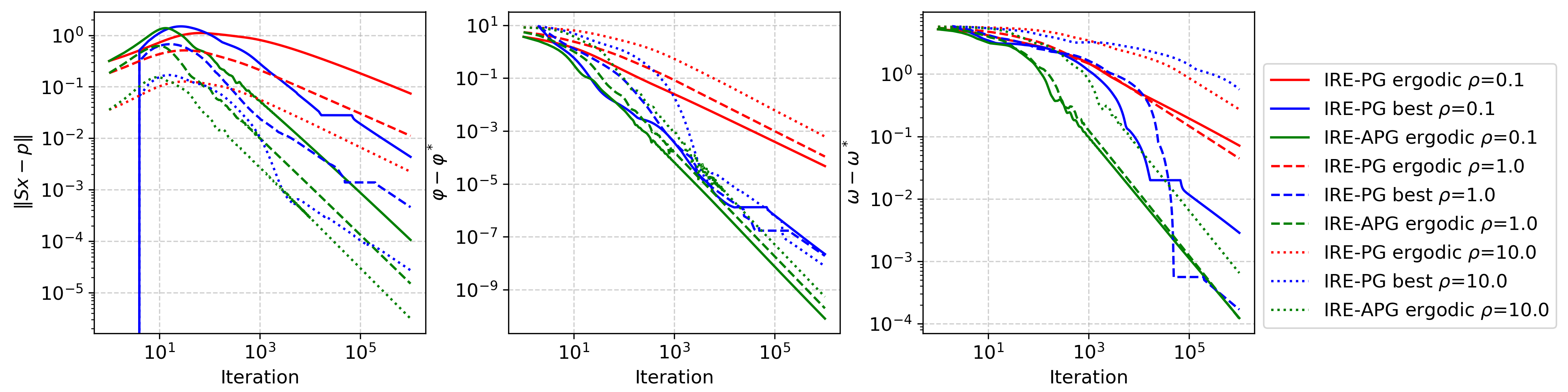}
\end{center}\vspace{-10pt}
\end{figure}
\section{Conclusion}\label{sec:conclusions}
In this paper, we explore two  algorithms, \ref{alg:IRE-PG} and \ref{alg:IRE-APG}, for solving simple bilevel problems, inspired by Solodov's first-order iterative regularization method. For both algorithms, we prove simultaneous rates of convergence for the inner and outer functions, a trade-off that depends on the choice of the regularization parameter. We demonstrate that in cases where the algorithms cannot be applied directly to the original bilevel problem, an equivalent surrogate problem can be solved instead. While the trade-off for \ref{alg:IRE-PG} extends naturally to this surrogate problem, the same does not hold for \ref{alg:IRE-APG}. It remains an open question whether the accelerated trade-off obtained by \ref{alg:IRE-APG} can be achieved by any method without additional assumptions on the problem structure. Additionally, it would be interesting to investigate whether convergence rate results to stationary points could be obtained for either method if 
 $\varphi$ or $\omega$  are not convex but retain a composite structure. We plan to explore these research questions in future work.

\appendix
\section{Proof of Lemma~\ref{lem:sum_bounds}}
\begin{proof}%
\begin{enumerate}[label=(\roman*)]
\item We consider different cases for $\beta$.\\ \underline{Case I}: $\beta\in [0,1)$. The function $k^{1-\beta}$ is monotonically increasing in $k$ 
   \begin{eqnarray*}
       \sum_{k=1}^{K}k^{1-\beta}&\leq& \int_1^{K+1}x^{1-\beta}dx %
       =\frac{(K+1)^{2-\beta}-1}{2-\beta}\leq \frac{(K+1)^{2-\beta}}{2-\beta}%
       \end{eqnarray*}
      \underline{Case II}: $\beta\in [1,2)$. The function $k^{1-\beta}$ is nonincreasing in $k$ and so   
\begin{eqnarray*}
       \sum_{k=1}^{K}k^{1-\beta}
       \leq (1 + \int_1^{K}x^{1-\beta}dx)=%
       (1+\frac{K^{2-\beta} - 1}{2-\beta})%
       &\leq& \frac{K^{2-\beta}}{2-\beta} 
       \leq \frac{(K+1)^{2-\beta}}{2-\beta}
       .\label{case3}
   \end{eqnarray*}
   \underline{Case III}: $\beta = 2$. Since $k^{-1}$ is decreasing in $k$ we have 
   \begin{eqnarray*}
       \sum_{k=1}^{K}k^{1-\beta}&=& \sum_{k=1}^{K}k^{-1}%
       \leq 1+\int_{1}^K x^{-1}dx
       \leq (1 + \log(K)).\label{case4}
   \end{eqnarray*}
    \underline{Case IV}: $\beta > 2$. The function  $k^{1-\beta}$ is monotonically decreasing in $k$ and so
    \begin{eqnarray*}
      \sum_{k=1}^{K}k^{1-\beta}
      &\leq& \big(1+\int_1^{K}x^{1-\beta}dx\big) = \big(1 + \frac{K^{2-\beta} - 1}{2-\beta}\big)
      \leq \frac{(\beta-1)}{\beta-2}. \label{casev5} 
    \end{eqnarray*}
   \item We consider different cases for $\beta$. \\
   \underline{Case I}:  $\beta\in [0,1)$. It follows from the fact that $k^{1-\beta}$ is increasing in $k$ that
    \begin{eqnarray*}
     \sum_{k=1}^{K}{k^{1-\beta}}
     &\geq& \big(1+\int_{1}^{K}x^{1-\beta}dx\big) %
     =\big(1+\frac{K^{2-\beta}-1}{2-\beta}\big) %
     \geq \frac{K^{2-\beta}}{2-\beta}\geq \frac{K^{2-\beta}}{2}.\label{casep1}
    \end{eqnarray*}
     \underline{Case II}: $\beta\in [1,2)$. The function $k^{1-\beta}$ is nonincreasing in $k$, so
     $$
     \sum_{k=1}^{K}{k^{1-\beta}}
     \geq K K^{1-\beta}=
     K^{2-\beta}{\geq \frac{K^{2-\beta}}{2}}.$$
     \underline{Case III}: $\beta = 2$. It follows from the fact that $k^{1-\beta}$ is decreasing in $k$ that
    \begin{eqnarray*}
     \sum_{k=1}^{K}{k^{1-\beta}}&=& \sum_{k=1}^{K}k^{-1} \geq \int_{1}^{K+1} x^{-1}dx =\log (K+1).
     \end{eqnarray*}
\item {For $K=1$ and $K=2$ the inequality trivially holds.}
For any $K\geq 3$
\begin{eqnarray}
        \sum_{k=1}^{K-1}\frac{2\log(k)}{(k+1)^3} +\frac{\log(K)}{(K+1)^2}
        &\leq &\sum_{k=2}^{K-1}\frac{2}{(k+1)^2}+\frac{1}{K+1}\nonumber\\
  &\leq& \frac{2}{9}+\int_{2}^{K-1}\frac{2}{(x+1)^2}dx+\frac{1}{(K+1)}\leq 1\nonumber
    \end{eqnarray}
    where the first inequality follows from $\log(k)\leq (k+1)$, and the second inequality follows from $1/(k+1)^2$ being nonincreasing. 
\end{enumerate}
\end{proof}
\section{Proof of Theorem~\ref{thm:rate_apg_constant_step} for backtracking step size}
\begin{proof}%
\begin{enumerate}[label=(\roman*).]
   \item  Using Lemma \ref{lem:rec_apg}(ii) we obtain that
\begin{equation}
\left(\sum_{k=1}^K\pi_k\right)(\omega(\bar{x}^K)-\omega^*)\leq \sum_{k=1}^K\pi_kz^k\leq \frac{1}{2}\|x^0 - x^*\|^2, \label{ratepiback} 
\end{equation}
where the inequalities follow from the convexity of $\omega$, and the definition of $\pi_k$, as in the derivation of \eqref{2.34}. Bounding the sum of $\pi_k$, we get
\begin{eqnarray}
  \sum_{k=1}^K\pi_k &=& \sum_{k=1}^{K-1}s^2_{k-1}(\sigma_{k}t_k-\sigma_{k+1}t_{k+1}) + \sigma_Kt_Ks^2_{K-1} \nonumber\\
  &\geq & \sum_{k=1}^{K-1} 
  \frac{(k+1)^2}{4}(\sigma_{k}t_k-\sigma_{k+1}t_{k+1}) + \sigma_Kt_K\frac{(K+1)^2}{4} \nonumber\\
  &=& \sum_{k=1}^{K}\frac{\sigma_kt_k}{4}((k+1)^2-k^2)%
  \geq \frac{1}{2\tilde{\alpha}_1}\sum_{k=1}^Kk^{1-\beta}%
  \geq \begin{cases}\frac{K^{2-\beta}}{4\tilde{\alpha}_1}, & \beta\in (0,2),\\
  \frac{\log(K+1)}{2\tilde{\alpha}_1} &\beta=2,\end{cases}\label{sum_piback}
\end{eqnarray}
where the first inequality follows from Lemma~\ref{lem:seq}, the second inequality is due to the fact that $t_k$ is nonincreasing in $k$,  the third inequality follows from $t_K\geq \min\{\frac{\gamma}{L_1+L_2}, \bar{t}\}=\tilde{\alpha}_1^{-1}$, and the last inequality follows from Lemma \ref{lem:sum_bounds}(ii). Substituting the bound \eqref{sum_piback} into \eqref{ratepiback}, we obtain the desired result.

\item For any $\beta\in (0, 2)$, 
it follows from Proposition~\ref{prop:vbound} that $\varphi(x^k)-\varphi(x^*)\leq \tilde{\alpha}_2 (\tilde{\alpha}_1\bar{t})^{-1}(k+1)^{-\beta}$. Thus, by the convexity of $\varphi$, and the definition of $\pi_k$, and $s_{k-1}\leq k$ from Lemma~\ref{lem:seq}, and \eqref{sum_piback} we have that
    \begin{eqnarray}
   \varphi(\bar{x}^K) - \varphi(x^*)&\leq& \frac{1}{ (\sum_{k=1}^{K}\pi_k)}\sum_{k=1}^{K}\pi_k(\varphi(x^k) - \varphi(x^*))\nonumber\\
    &\leq& \frac{\tilde{\alpha}_2(\tilde{\alpha}_1\bar{t})^{-1}}{K^{2-\beta}(4\tilde{\alpha}_1)^{-1}}(\sum_{k=1}^{K-1} k^2(\sigma_kt_k-\sigma_{k+1}t_{k+1})k^{-\beta}+ K^{2-\beta}\sigma_Kt_K)\nonumber\\
    &=& \frac{4\tilde{\alpha}_2\bar{t}^{-1}}{K^{2-\beta}}\sum_{k=1}^{K} \sigma_kt_kk^{-\beta}(k^2-(k-1)^2)%
    \leq \frac{8\tilde{\alpha}_2}{K^{2-\beta}}\sum_{k=1}^{K}k^{1-2\beta}.\label{2.34c}
    \end{eqnarray}  
    Applying  Lemma \ref{lem:sum_bounds}(i) and $(K+1)^{2-2\beta}\leq 4K^{2-2\beta}$ to \eqref{2.34c}, we obtain the desired result for $\beta\in(0,2).$ 
        Similarly, for the  case where $\beta=2$, it follows from Proposition~\ref{prop:vbound} that $\varphi(x^k)-\varphi(x^*)\leq \frac{\tilde{\alpha}_3+\tilde{\alpha}_4\log(k)}{\tilde{\alpha}_1\bar{t}(k+1)^2}$, and thus,
    \begin{eqnarray}
&&\varphi(\bar{x}^K) - \varphi(x^*)\leq\frac{1}{\sum_{k=1}^{K}\pi_k}\sum_{k=1}^{K}\pi_k(\varphi(x^k) - \varphi(x^*))\nonumber\\
    \nonumber\\
    &&\leq \frac{(\tilde{\alpha}_1\bar{t})^{-1}2\tilde{\alpha}_1}{\log(K+1)}\left(\sum_{k=1}^{K-1}\frac{(\sigma_kt_k-\sigma_{k+1}t_{k+1})k^2(\tilde{\alpha}_3+\tilde{\alpha}_4\log(k))}{(k+1)^2}\right.\nonumber\\
    &&\qquad\left.+\frac{\sigma_{K}t_{K}K^2(\tilde{\alpha}_3+\tilde{\alpha}_4\log(K))}{(K+1)^2}\right)
    \nonumber\\    &&=\frac{2\tilde{\alpha}_3\bar{t}^{-1}}{\log(K+1)}\sum_{k=1}^{K}\sigma_kt_k\left(\frac{k^2}{(k+1)^2}-\frac{(k-1)^2}{k^2}\right)\nonumber\\
    &&\quad+\frac{2\tilde{\alpha}_4\bar{t}^{-1}}{\log(K+1)}\sum_{k=1}^{K}\sigma_kt_k\left(\frac{k^2\log(k)}{(k+1)^2}-\frac{(k-1)^2\log(k-1)}{k^2}\right)\nonumber\\
    &&\leq \frac{2\tilde{\alpha}_3}{\log(K+1)}\sum_{k=1}^{K}\frac{(2k^2-1)}{k^4(k+1)^2}\nonumber\\
    &&\quad+\frac{2\tilde{\alpha}_4}{\log(K+1)}\sum_{k=1}^{K}\left(\frac{\log(k)}{(k+1)^2}-\frac{(k-1)^2\log(k-1)}{k^4}\right)
     \nonumber\\
     \nonumber\\
    &&\leq \frac{2(4\tilde{\alpha}_3+\tilde{\alpha}_4)}{\log(K+1)}.\nonumber%
    \end{eqnarray}
    where the second inequality is due to the definition of $\pi_k$, $s_{k-1}\leq k$ from Lemma~\ref{lem:seq}, and \eqref{sum_piback}, the third inequality is due to $t_k\leq \bar{t}$, and the last inequality  follows the same arguments as those leading to \eqref{ratev6} in the proof of Theorem~\ref{thm:rate_apg_constant_step}.
\end{enumerate}
\end{proof}

\section*{Acknowledgments}
The authors are grateful to the anonymous referees and the associate editor for their insightful comments and suggestions, which significantly improved the quality of this manuscript.

\bibliographystyle{plain}
\bibliography{iucr_arxiv}

\end{document}